\documentclass{amsart}

\usepackage[english]{babel}
\usepackage{booktabs,url}
\usepackage{graphicx,subcaption,tikz}
\usepackage{comment}

\newtheorem{thm}{Theorem}[section]
\newtheorem{prop}[thm]{Proposition}
\newtheorem{lem}[thm]{Lemma}
\newtheorem{cor}[thm]{Corollary}
\newtheorem{conj}[thm]{Conjecture}

\theoremstyle{definition}  

\newtheorem{defn}[thm]{Definition}
\newtheorem{ex}[thm]{Example}
\newtheorem{ques}[thm]{Question}

\newcommand\eqsubref[1]{(\subref{#1})}

\DeclareMathOperator\bridge{bridge}
\DeclareMathOperator\braid{braid}
\newcommand\bb[2]{b_{#2}(#1)}
\newcommand\bbb[1]{b({#1})}
\usepackage[pic,knot,graph,curve]{xy}

\begin{document}



\title{A spectrum connecting the bridge index and the braid index}

\author{Margaret Doig}
\address{Department of Mathematics, Creighton University}
\email{margaretdoig@creighton.edu}
\author{Chase Gehringer}
\address{Department of Mathematics, Creighton University}
\email{chasegehringer@creighton.edu}


\begin{abstract}
    We study the booklink, a braid-like embedding with local maxima and minima, and the bridge-braid spectrum of a link, which captures the smallest number of braid-strands in a booklink with a prescribed number of critical points. This spectrum spans the gap between the classical bridge and braid indices. We apply a foliation theory argument to provide a formula for the spectra of both split and composite links. We then generate a table for the spectra of all prime knots up to 9 crossings.
\end{abstract}

\maketitle

\section{Introduction}

Any student, in their first exposure to braids, will have played the game of Alexander's Theorem, starting with a knot diagram and painstakingly converting it, arc by arc, into a braid diagram. See Figure~\ref{fig:generic_knot}: we start with the diagram on the left, happily floating in front of the $z$-axis, and we check to see if the knot wraps smoothly around the axis in one direction as required. It does, with a few exceptions, the solid dots where the braid reverses direction. This knot has four arcs, two going in one direction and two in the other. To fix it, we first grab an arc going the wrong direction, then we slip it over the axis to resolve it, as in the second picture. Next, we grab the remaining arc to slip it over the axis \dots but we can't! It gets hung up on one of the crossings --- we can, though, slip half of it over, giving the third picture, and then come back for the last and slip it over, too, just the other way. 

\begin{figure}[h]\centering
    \begin{subfigure}{0.2\textwidth}\centering
        \includegraphics[scale=0.45]{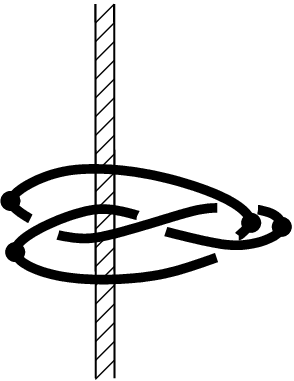}\caption{}\label{fig:fig8-20}
    \end{subfigure}
    \begin{subfigure}{0.2\textwidth}\centering
        \includegraphics[scale=0.45]{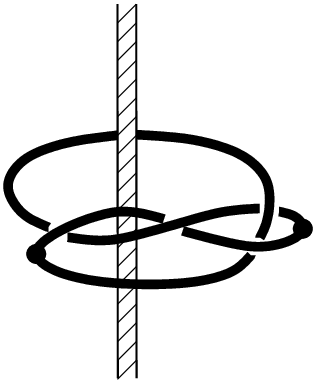}\caption{}\label{fig:fig8-11}
    \end{subfigure}
    \begin{subfigure}{0.2\textwidth}\centering
        \includegraphics[scale=0.45]{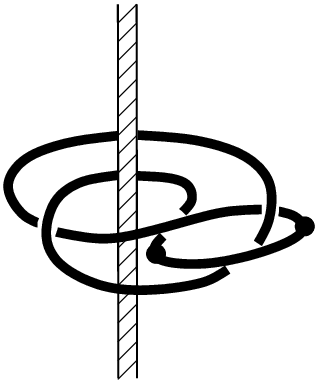}\caption{}\label{fig:fig8-12}
    \end{subfigure}
    \begin{subfigure}{0.25\textwidth}\centering
        \includegraphics[scale=0.45]{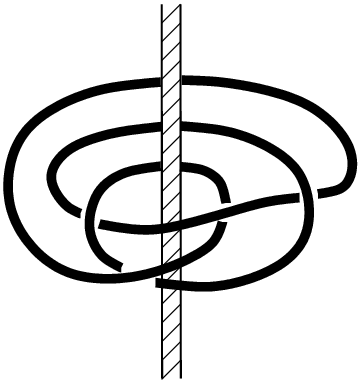}\caption{}\label{fig:fig8-03}
    \end{subfigure}
\caption{The figure 8 knot: a visualization of the application of Alexander's Theorem to convert a plat presentation to a braid.}\label{fig:generic_knot}
\end{figure}

Alexander gave us a recipe for resolving the knot into a braid, but he didn't tell us how many loops it would take, or how difficult it would be. These questions came later, and were answered later, for example, by all the powerful polynomial invariants which bound or identify the braid index. Even these invariants, though, only tell us the final result, how many loops we will need. There is much they do not encapsulate about the process which we all intuitively came to understand: for example, we can always find a first arc which may be completely resolved in one swipe (this is a result of the fact that any link can be put in plat form with at least one strand free of any crossings). Further, though, we all realized that, once the going gets tough (e.g., it takes two loops to resolve the next arc), then the going stays tough (e.g., any further arc will also require at least two loops). We have not previously had any way to describe or study this fact mathematically, but we do so here using the language of booklinks, a way to interpolate between the plat picture of the left and the braid picture of the right, as well as to study the intermediate steps as intermediate states. 

Booklinks were developed by Aranda, Binns, and one of the authors in~\cite{arandaXXXbooklink}, and studying them draws heavily on the work of Birman and Menasco studying braids by examining foliations of surfaces in $S^3$, especially the series~\cite{birman1992studying1, birman1991studying, birman1993studying, birman1990studying, birman1992studying5, birman1992studying}, as well as later work on braids in open book decompositions such as Ito and Kawamuro's extensions to braids in arbitrary open book manifolds~\cite{ito2014open,ito2014operations}. A comprehensive text on the applications of foliation theory to braids is LaFountain and Menasco's~\cite{lafountain2017braid}. Recently, Menasco and Solanki have applied foliation techniques to study plats, which lie at the other extreme of the link world~\cite{solanki2023studying,menasco2024studying}. Booklinks span this gap between braids and plats, and we adapt the techniques of foliation theory to study them.

\begin{figure}\centering
    \begin{subfigure}{0.2\textwidth}\centering
        \includegraphics{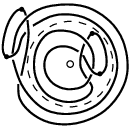} 
        \caption{}\label{fig:split_good}
    \end{subfigure}
    \begin{subfigure}{0.2\textwidth}\centering
        \includegraphics{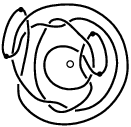} 
        \caption{}\label{fig:split_bad}
    \end{subfigure}
    \begin{subfigure}{0.2\textwidth}\centering
        \includegraphics{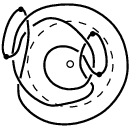} 
        \caption{}\label{fig:composite_good}
    \end{subfigure}
    \begin{subfigure}{0.2\textwidth}\centering
        \includegraphics{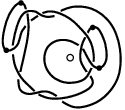} 
        \caption{}\label{fig:composite_bad}
    \end{subfigure}
    \caption{(\subref{fig:split_good}-\subref{fig:split_bad}) Two booklink representatives of a split link, one which is split as a booklink (see dotted circle) and one which is not; (\subref{fig:composite_good}-\subref{fig:composite_bad}) two representatives of a composite link, one which is composite as a booklink and one which is not.}\label{fig:split_composite}
\end{figure}

The first goal of this paper is to conceptualize the behavior of booklinks under split union and composition. A booklink representing a split link may be embedded as in Figure~\ref{fig:split_good} so that there is a clear sphere separating the two components, or as in Figure~\ref{fig:split_bad}, where the components are not clearly separable. We call the former a \emph{split booklink} (see Section~\ref{sec:split_composite}). The same can be said regarding the composite case. The question is then how to transform an arbitrary booklink representative of a split or composite link into a split or composite booklink. The Booklink Markov Theorem~\cite[Theorem~1.2]{arandaXXXbooklink} tells us that any two booklinks in the same isotopy class are connected by a sequence of booklink isotopies and stabilizations/destabilizations, but these operations can profoundly alter the booklinks and their invariants. Instead, we demonstrate below a method to do this transformation, using only booklink isotopies and exchange moves. 

\begin{thm}[Split Booklink Theorem]\label{thm:split}
Let $L$ be a split link in $S^3$ and $\lambda$ a booklink representative. Then there is a finite sequence of booklink isotopies and exchange moves which transforms $\lambda$ into $\lambda'$, a split booklink representative of $L$.
\end{thm}

\begin{thm}[Composite Booklink Theorem]\label{thm:composite}
Let $L$ be a composite link in $S^3$ and $\lambda$ a booklink representative. Then there is a finite sequence of booklink isotopies and exchange moves which transforms $\lambda$ into $\lambda'$, a composite booklink representative of $L$.
\end{thm}

These two results mirror and extend two previous sets of results: Birman and Menasco originated this conversation with the Split and Composite Braid Theorems~\cite[p.~117--118]{birman1990studying}, which also substitute exchange moves for stabilization/destabilization in order to preserve the braid index during this transformation. More recently, Menasco and Solanki continued this pattern of study with the Split and Composite Plat Theorems~\cite[Theorems~1--2]{menasco2024studying}, which replace stabilizations/destabilizations with a set of specialized moves on plat diagrams that preserve the bridge index (pocket moves, flip moves, and double coset moves).

Our second goal is to understand the bridge-braid spectrum, a new invariant for links that spans the gap between the classical bridge and braid indices. At its heart, it traces the conversion of a plat diagram into a braid via Alexander's Theorem as shown in Figure~\ref{fig:generic_knot}. 

Aranda, Binns, and Doig introduced the spectrum in~\cite{arandaXXXbooklink}. In order to study it, we first apply the theorems above to derive the spectra of split and composite links, encompassing both Birman and Menasco's Braid Index Theorem~\cite[p.~118]{birman1990studying} and Schubert's Bridge Index Theorem~\cite{schubert1956knoten}.

\begin{cor}\label{cor:splitspectrum}
    Let $L = L_1 \cup L_2$ be a link. Then 
    \[\bb{d}{L} \geq \min_{d_1 + d_2 = d} \big[\bb{d_1}{L_1}+\bb{d_2}{L_2}\big]\]
    with equality for a split link.
\end{cor}

\begin{cor}\label{cor:compositespectrum}
    Say $L = L_1 \# L_2$ is a composite link. Then
    \[\bb{0}{L} = \bb{0}{L_1} + \bb{0}{L_2} - 1.\]
    If $0 < d < \bridge(L_1) + \bridge(L_2) - 1$,
    \[\bb{d}{L} = \min_{d_1+d_2 = d+1} \big[\bb{d_1}{L_1} + \bb{d_2}{L_2}\big].\]
    Otherwise,
    \[\bb{d}{L} = 0.\]
\end{cor}

We then collect results on the behavior of the spectra and complete its calculation for prime knots through 9 crossings, which we present in Table~\ref{table:knot_spectra}:

\begin{thm}\label{thm:knot_spectra}
Table~\ref{table:knot_spectra} lists the bridge-braid spectra for all knots through 9 crossings.
\end{thm}


\subsection*{Outline}
In Section~\ref{sec:defn}, we define booklinks, split and composite booklinks, and the bridge-braid spectrum of a link. In Section~\ref{sec:split}, we show how to derive the spectrum of a split link from the spectra of its component links; in Section~\ref{sec:composite}, we do the same for composite links. Finally, in Section~\ref{sec:spectra}, we demonstrate some general results about spectra, produce 25 booklink diagrams, and calculate the spectra of prime knots through 9 crossings. In Section~\ref{sec:futurework}, we summarize possible future directions and outstanding questions.

\subsection*{Acknowledgments}
Thanks to Rom\'an Aranda for helpful comments and feedback.

\section{Definitions}\label{sec:defn}

\subsection{Booklinks and surfaces}
A \emph{knot} in $S^3$ is a smooth embedding of $S^1$ into $S^3$, up to isotopy, and a \emph{link} is a smooth embedding of $\cup^n S^1$. A special kind of link is a \emph{braid} in $R^3$, which wraps smoothly in one direction around an axis. More precisely, when the knot is given in cylindrical coordinates, where $\theta$ is the angle, then $\theta$ has no critical points. In \cite{arandaXXXbooklink}, Aranda, Binns, and Doig extended the definition of a braid to a \emph{booklink}, which is a link where $\theta$ is allowed to have non-degenerate critical points. See Figure~\ref{fig:generic_knot}, where the solid are critical points. We will regularly consider booklinks up to \emph{booklink isotopy}, which is an isotopy through booklinks; in particular, the number of critical points remains constant. When we reference a \emph{booklink type}, we mean an equivalence class of booklinks under booklink isotopy. These concepts can be extended to open book 3-manifolds: if a manifold has an open book decomposition, then a braid is a link which is never tangent to a page, and a booklink is a link which is only tangent to pages in a finite set of non-degenerate critical points.

We will study booklinks and surfaces related to them in part by examining the role of $\theta$ in cylindrical coordinates. The \emph{page} $H_\theta$ is the set of all points away from the $z$-axis which have angle $\theta$, and a \emph{leaf} on some surface is the intersection of the surface with some $H_\theta$. We may consider the \emph{foliation} of a surface, the set of all such leaves on the surface. We will study foliations and critical points of surfaces below. 

\subsection{Invariants}
There are two numbers associated to a given booklink, $d$, the number of pairs of non-degenerate critical points, and $n$, the minimum geometric intersection number of the booklink with any page $H_{\theta}$ of the open book decomposition; sometimes we will fix a page that realizes this minimum and call it $H_{min}$. We call this booklink a \emph{$(d,n)$-representative} of its booklink type. The number $d$ is invariant under booklink isotopy, and so we extend it to an invariant of the booklink isotopy class itself and call it the \emph{bridge number}. The number $n$, however, may change under booklink isotopy, and so we designate the \emph{braid number} of a booklink type to be the minimum $n$ achievable in the isotopy class, that is, the minimum number of intersections of any booklink in the class with any page $H_\theta$. With this terminology, a braid is a booklink with $d=0$, equivalently, a booklink with only $(0,n)$-representatives. At the other extreme, a plat is a booklink with braid number $0$, i.e., one which has $(d,0)$-representatives (although some representatives in the booklink type may have larger $n$). In Figure~\ref{fig:generic_knot}, the four booklinks are $(2,0)$-, $(1,1)$-, $(1,2)$-, and $(0,2)$-representatives from left to right; all achieve the braid number for their booklink type. 

The \emph{bridge-braid spectrum} of a link type $L$ is the sequence \[\big\{\bb{0}{L}, \bb{1}{L}, \cdots, \bb{d}{L}, \cdots\big\}\] where $\bb{d}{L}$ is the minimum braid index for any booklink representative of $L$ that has bridge index $d$, i.e., the minimum $n$ such that there is a $(d,n)$-representative for the booklink:
\[\bb{d}{L} = \min_{\substack{\lambda \in L \\ \bridge(\lambda) = d}} \braid(\lambda).\]
Note that this sequence is eventually 0, and so we often abbreviate it by omitting all zeroes after the first one.

\subsection{Split and composite links}\label{sec:split_composite}
Two special types of links we will study are split and composite.

A link $L$ is a \emph{split link} if there exists an embedded sphere that separates $L$ into two components, called $L_1$ and $L_2$. We write $L$ as $L_1 \cup L_2$. In particular, a braid or a booklink represents a split link if there exists such a splitting sphere. In addition, some booklinks are split in an especially obvious way where we may easily see the splitting sphere, and we call these \emph{split booklinks}, where the splitting sphere has exactly two intersections with the $z$-axis, and both are transverse. See Figure~\ref{fig:split_composite}. The definition encapsulates both split braids and split plats, which satisfy the equivalent requirements.

A link $L$ is a \emph{composite link} if there exists an embedded sphere that intersects the link transversely in two points and divides the link into two non-trivial arcs $L_1'$ and $L_2'$, that is, if we close $L_i'$ by adding an arc in the sphere to form a link $L_i$, then it will be non-trivial. We then write $L$ as $L_1 \# L_2$. In particular, a braid or a booklink may be a composite link if there exists such a splitting sphere. In addition, some braids and booklinks are composite in an especially obvious way where we may easily see the splitting sphere, and we call these \emph{composite booklinks}, where the splitting sphere has exactly two intersections with the $z$-axis, and both are transverse. See Figure~\ref{fig:split_composite}. Once more, these include the classical concepts of composite braids and composite plats.

Observe that the notation $\lambda_1 \#\lambda_2$ is not well-defined in the same way that $L_1 \# L_2$ is. Given two oriented link types, the composite is independent of the location where the sum is accomplished because, if we wish to sum the sublinks in a different location, we merely need to shrink one summand into a tiny ball and slide it along the other summand to the new location. The same is true for braids. For booklinks, by contrast, a connected sum is more complicated: we cannot slide a summand over a critical point without changing the booklink class; additionally, if we perform the sum without care, we may alter the braid and/or bridge index. The sum $\lambda_1\#\lambda_2$ is only well-defined up to the critical point or arc between critical points where the sum occurs. 

\subsection{Exchange Moves}
When Birman and Menasco proved the Split Braid Theorem and the Composite Braid Theorem, they defined the \emph{exchange move} to alter a braid without affecting the number of strands or its isotopy class~\cite{birman1990studying}. When we prove the corresponding theorems for booklinks, we will use the same tool extended to booklinks, the exchange move shown in Figure~\ref{fig:exchange}. Each solid arc shown in the figure may in fact represent a thin tube inside which live multiple strands, and these strands may be locally knotted inside the tube. Observe that the exchange move, like the booklink isotopy, preserves both bridge and braid index.

\begin{figure}
\begin{subfigure}[c]{0.45\textwidth}\centering
\includegraphics[scale=0.7]{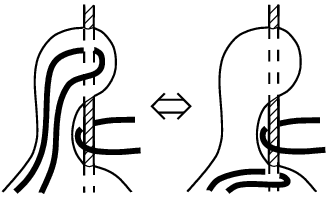}\label{fig:exchange_side}
\end{subfigure}
\begin{subfigure}[c]{0.45\textwidth}\centering
\includegraphics[scale=0.7]{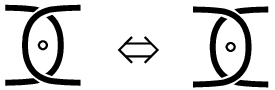}\label{fig:exchange_overhead}
\end{subfigure}
\caption{An exchange move, from the side and from above.}\label{fig:exchange} 
\end{figure}

\section{Split Links} \label{sec:split}

Theorem~\ref{thm:split} is a generalization of Birman and Menasco's result for braids~\cite[Split Braid Theorem]{birman1990studying}, that any braid which is split can be transformed by a set of reasonable moves into a braid which is split in a very obvious way.



\begin{figure}
\begin{subfigure}{\textwidth}\centering
    \includegraphics[scale=0.3]{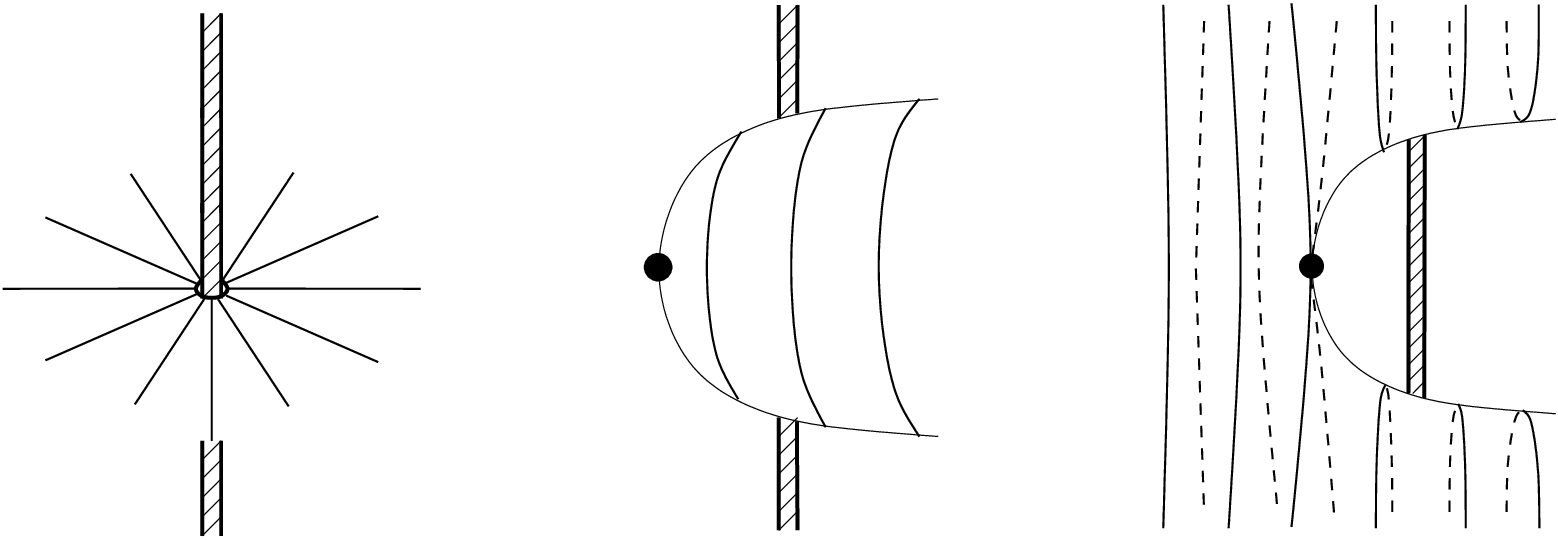}\caption{}\label{fig:singularities}
\end{subfigure}
\par\bigskip
    \begin{subfigure}{\textwidth}\centering
        \includegraphics[scale=0.2]{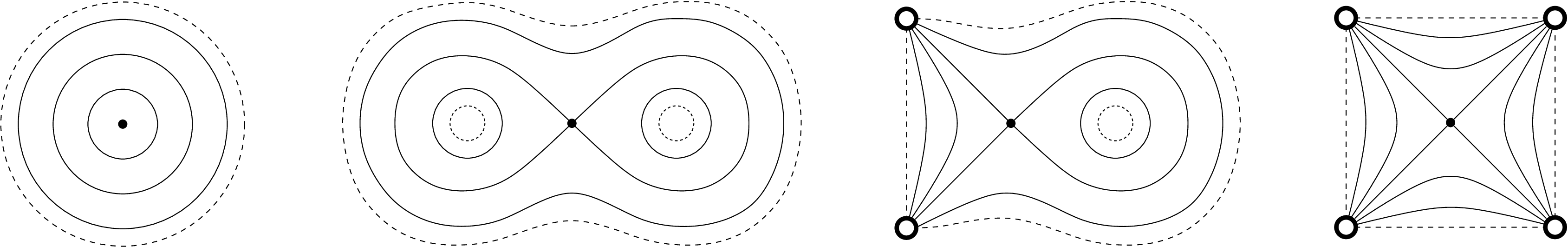}\subcaption{}\label{fig:tiles}
    \end{subfigure}
    \caption{\eqsubref{fig:singularities} A foliated closed surface may include three types of singular points: elliptic points, extremal singularities, and hyperbolic singularities. \eqsubref{fig:tiles} It may be tiled by neighborhoods of the latter two, which come in four flavors.}\label{fig:foliation}
\end{figure}

To prove the theorem, we will first identify a splitting sphere $S$ for the booklink $\lambda$ and, by Lemma~\ref{lem:genposn}, we will put it into general position so that the foliation induced by the open book decomposition of $S^3$ is sufficiently nice. We will then induct on the number of elliptic points. By Lemma~\ref{lem:scc}, we will first eliminate SCCs in the foliation; by Lemma~\ref{lem:nonessential}, we will alter the surface via booklink isotopy to cancel pairs of elliptic points identified by the presence of a particular type of leaf. Finally, we will apply an Euler characteristic argument in Lemma~\ref{lem:exchange} to identify exchange moves to alter the splitting sphere until it has exactly 2 elliptic points, at which point the theorem is satisfied.

\begin{defn}
    We say a surface $\Sigma$ in an open book manifold with a booklink $\lambda$ is in \emph{foliated general position} if:
 \begin{itemize}
  \item $\Sigma$ intersects the braid axis in a finite number of points, where it is transverse and has a radial foliation, which we call \emph{elliptic points};
  \item $\Sigma$ is transverse to all but finitely many pages, each of which has exactly one singularity;
  \item $\Sigma$ is tangent to the surface at only non-degenerate singularities, at either \emph{extremal singularities} (minima or maxima with respect to the open book) or \emph{hyperbolic singularities} (saddles with respect to the open book);
  \item Any intersections between $\lambda$ and $\Sigma$ occur away from singularities of $\Sigma$ and critical points of $\lambda$.
 \end{itemize}
\end{defn}

Observe that this definition is similar to the concept of an open book foliation described in \cite[Definition~3.3]{arandaXXXbooklink}, with the exception that we allow extremal singularities in the interior of the surface and restrict the intersections with $\lambda$. 

\begin{lem}\label{lem:genposn}
A closed surface in an open book manifold containing a booklink may be put into foliated general position via booklink isotopy.
\end{lem}

\begin{proof}
    We may assume by a general position argument that the surface intersects the braid axis in a finite number of points where it is transverse and has a radial foliation. Likewise, we may perturb the surface to ensure that it is transverse to all but finitely many pages and that each singular page has exactly one singularity. See Figure~\ref{fig:singularities}. Finally, if the booklink and the surface intersect, we may perturb the surface or the booklink to move the intersections off the singular leaves and the critical points.
\end{proof}

Next, we observe that the foliation is made up of several distinct types of pieces.

\begin{defn}
    The complement of the singularities (both elliptic points and extremal/hyperbolic singularities) and the singular leaves they pass through is a finite number of disconnected product neighborhoods of nonsingular leaves. We may identify a single nonsingular leaf in the middle of each of these product neighborhoods and envision dividing the surface along these chosen leaves into regions we call \emph{tiles}.
\end{defn}

\begin{lem}\label{lem:tiles}
    If a closed surface $\Sigma$ is in foliated general position, then its foliation can be decomposed into the four tiles of the type found in Figure~\ref{fig:tiles}.
\end{lem}

\begin{proof}
    Lemma~2.4 of \cite{lafountain2017braid} allows us to decompose a surface into tiles if it is in the general position described in Lemma~\ref{lem:genposn}.

    An extremal singularity has a neighborhood consisting of a product of SCCs and so can occur only in one way. 

    The neighborhood of a hyperbolic singularity includes four singular leaves. We may classify them based on the number of elliptic points involved. All four singular leaves may terminate at (distinct) elliptic points, forming an \emph{h-tile}. On the other hand, two may terminate in elliptic points and the other two form a closed curve; note that, since the page is a half-plane, the closed curve separates it into two components, only one of which may touch the binding, and so the arcs terminating in elliptic points cannot be diagonally across from one another at the hyperbolic singularity. Finally, the singular leaves may also form two closed curves, which, again, may only take the form shown.
\end{proof}

Now, we begin to simplify our splitting sphere $S$ by eliminating SCCs. 

\begin{defn}
    There are three types of closed curves we may encounter in a surface in foliated general position. If $\alpha$ is a SCC, then it lives in a product neighborhood of simple closed curves bounded at each end by a singular leaf. We name the SCCs based on whether this neighborhood contains two extremal singularities (\emph{type E-E}), two hyperbolic singularities (\emph{type H-H}), or one of each (\emph{type E-H}).
\end{defn}

\begin{figure}\centering
    \begin{subfigure}{0.3\textwidth}\centering
        \includegraphics[scale=0.7]{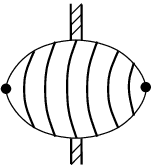}
        \caption{}\label{fig:scc_ee}
    \end{subfigure}\qquad
    \begin{subfigure}{0.6\textwidth}\centering
        \includegraphics[scale=0.7]{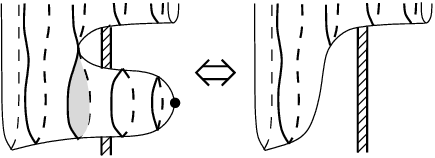}
        \caption{}\label{fig:scc_eh}
    \end{subfigure}
    \par\medskip
    \begin{subfigure}{0.6\textwidth}\centering
        \includegraphics[scale=0.7]{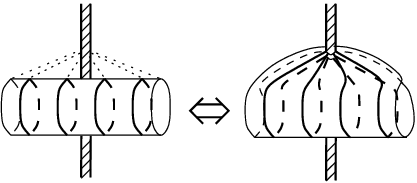}
        \caption{}\label{fig:scc_hh}
    \end{subfigure}\qquad
    \begin{subfigure}{0.3\textwidth}\centering
        \includegraphics[scale=0.7]{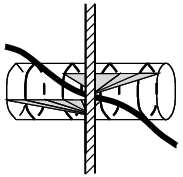}
        \caption{}\label{fig:scc_hh-v2}
    \end{subfigure}
    \caption{Three types of SCCs and their resolutions, \eqsubref{fig:scc_ee} type E-E, \eqsubref{fig:scc_eh} type E-H, \eqsubref{fig:scc_hh}-\eqsubref{fig:scc_hh-v2} type H-H and a variation.}\label{fig:scc}
\end{figure} 

\begin{lem}\label{lem:scc}
    If a closed, connected surface $\Sigma$ is in foliated general position, then its SCCs are types E-E, E-H, or H-H, as pictured in Figure~\ref{fig:scc}. If $\Sigma$ has a SCC of type E-E, then it is a sphere without elliptic points. If $\Sigma$ is not punctured by the booklink $\lambda$, then any SCC of type E-H or H-H may be eliminated by a booklink isotopy (and, for type E-H, does not add any elliptic points). 
\end{lem}

Observe that, if the SCC is type E-E and so the splitting sphere $S$ is a sphere without elliptic points, then we may apply a booklink isotopy so that Theorem~\ref{thm:split} is satisfied.

\begin{proof}
    We may ``flow'' a SCC $\alpha$ both forwards and backwards through the open book until it encounters a singular leaf in each direction, which gives a product neighborhood or cylinder of SCCs. This flow can never encounter an elliptic singularity, so it must be one of the three types listed.
 
    In the E-E case, where $\alpha$ encounters an extremal singularity in both directions, the surface $\Sigma$ is in fact a sphere disjoint from the braid axis as in Figure~\ref{fig:scc_ee}. We may easily identify a small region of it and push it through the axis, adding 2 elliptic points (and possibly a pair of hyperbolic singularities).
 
 	If we encounter an extremal singularity in one direction and a hyperbolic one in the other, i.e., the E-H case, then we follow Figure~\ref{fig:scc_eh}, which reproduces the method of \cite[Figure~2.7 and Lemma~2.1]{lafountain2017braid}: we may isotope the protruding finger of surface (along with any booklink inside it) back past the page with the hyperbolic critical point. That is, the curve $\alpha$ will flow to a closed curve inside the singular leaf which bounds a disk $D_{\theta}$ in its page $H_{\theta}$ as well as a disk $D_\Sigma$ in $\Sigma$ which contains $\alpha$ and the extremal critical point. We will assume inductively that the interior of $D_{\theta}$ is disjoint from $\Sigma$, and so we may use the solid ball bounded by $D_{\theta}$ and $D_\Sigma$ and construct a booklink isotopy replacing $\Sigma$ by $\Sigma-D_\Sigma \cup D_{\theta}$. If there was any portion of the booklink inside this ball, it may be pushed across $D_{\theta}$ while preserving the critical points, so the isotopy is a booklink isotopy. 
 
    Finally, for the H-H case, we have a cylinder of SCCs which terminates at a hyperbolic singularity in both directions. We may resolve it as in Figure~\ref{fig:scc_hh}, which follows the method of \cite[Figure~2.14 and Lemma~2.3]{lafountain2017braid}): We identify a continuous family of paths, each connecting an SCC in the cylinder through its page to some fixed point on the braid axis; if this family is disjoint from the booklink, we use it to isotope a piece of the surface through the axis, adding two elliptic points and eliminating the whole cylinder of SCCs. Observe that we are not concerned if some other portion of the surface intersects this family of paths, as the isotopy will push it smoothly through the binding as well. If the booklink itself interrupts the paths, we can divide the cylinder of SCCs into multiple sub-cylinders which we resolve separately, as in Figure~\ref{fig:scc_hh-v2}. Since this isotopy is now conducted away from the booklink, it is in fact a booklink isotopy.
\end{proof}

Now, since our splitting sphere $S$ has no SCCs and thus no extremal singularities, it must be tiled by $h$-tiles glued together along nonsingular leaves. We now begin our induction by canceling pairs of elliptic points.

\begin{defn}
    Consider a nonsingular leaf terminating in two elliptic points: this leaf is a curve in some page $H_{\theta}$ which touches the binding in two points and divides the page into two disks. If one of these disks is disjoint from the booklink, then we call it \emph{nonessential}, else we call it \emph{essential}.
\end{defn}

\begin{figure}
    \includegraphics[scale=0.6]{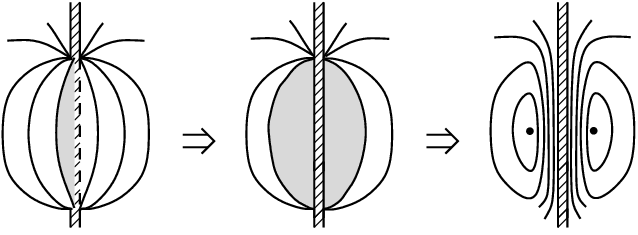}
    \caption{A nonsingular leaf between two elliptic points divides its page into two disks. If one of them avoids the booklink, then the surface may be pushed along that disk to remove the pair of elliptic points.}\label{fig:nonessential}
\end{figure}

\begin{lem}\label{lem:nonessential}
    If the splitting sphere $S$ has a nonessential leaf, we may alter it by booklink isotopy and remove a pair of elliptic points.
\end{lem}

\begin{proof}
    Consider a nonessential leaf as in Figure~\ref{fig:nonessential}. It divides its page into two disks, one of which misses the booklink (in the figure, the shaded inner disk), and so we may isotope $S$ along that disk. This removes a pair of elliptic points, although it reintroduces a pair of extremal singularities and associated SCCs. These SCCs cannot be of type E-E as there are too many elliptic points, so they will be of type E-H, and we may then resolve them by a new application of the E-H method of Lemma~\ref{lem:scc} without reintroducing any elliptic points. Once more, the entire process occurs away from the booklink.
\end{proof}

If the application of Lemma~\ref{lem:nonessential} leaves us with more than two elliptic points, then we will identify a possible exchange move on the booklink which will convert an essential leaf to nonessential and allow us to cancel another pair of elliptic points.

\begin{figure}\centering
    \begin{subfigure}{\textwidth}\centering
        \includegraphics[scale=0.35]{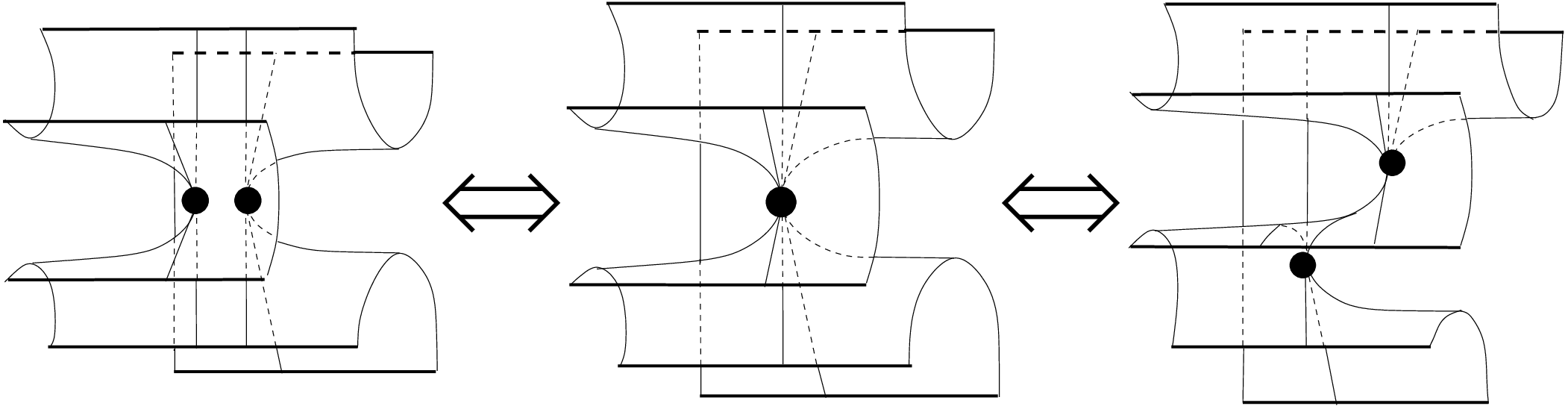}\caption{}\label{fig:monkeysaddlesurface}
    \end{subfigure}
    \par\medskip
    \begin{subfigure}{\textwidth}\centering
        \includegraphics[scale=0.2]{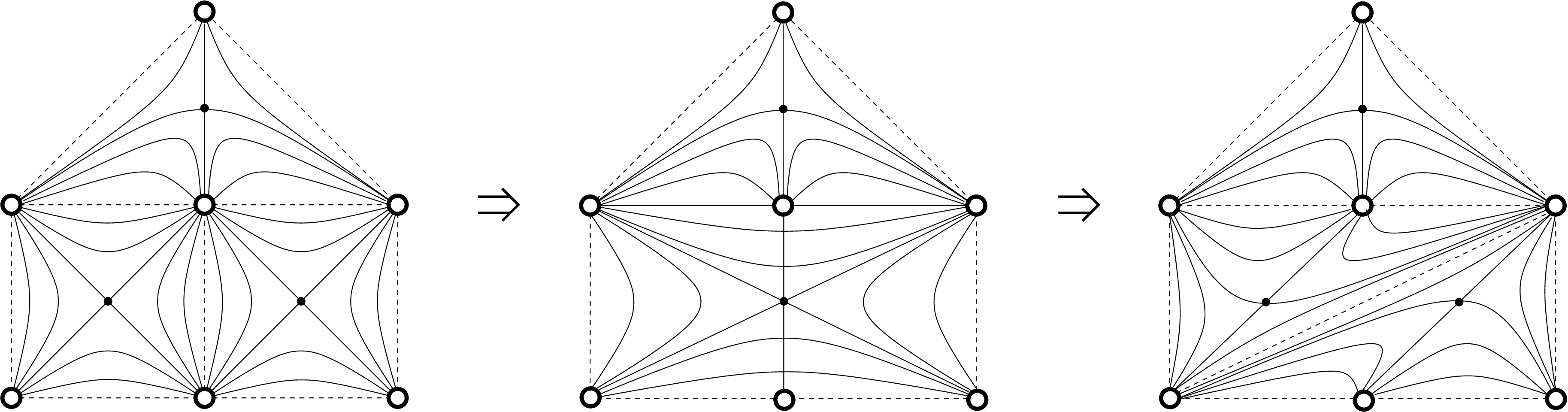}\caption{}\label{fig:monkeysaddlemerge}
    \end{subfigure}
    \caption{\eqsubref{fig:monkeysaddlesurface} Two adjacent saddle points of the same orientation may merge and separate in a local isotopy. \eqsubref{fig:monkeysaddlemerge} A possible use of this isotopy to alter the tiling around a degree-3 elliptic point.}\label{fig:monkeysaddle}
\end{figure}

\begin{lem}\label{lem:exchange}
    Let $S$ be in foliated general position without SCCs and where all nonsingular leaves are essential. If there are more than 2 elliptic points, then we may perform an exchange move which converts a leaf from essential to nonessential leaf without altering the surface.
\end{lem}

\begin{proof}
    Consider the multigraph with a vertex for each elliptic point and an edge for each product family of nonsingular leaves (this is well-defined because there are more than 2 elliptic points). We repeat the Euler characteristic argument of \cite[Lemma~2]{birman1990studying}: the multigraph gives a decomposition of the sphere, so $F-E+V = 2$; additionally, since there are more than 3 elliptic points, the graph cuts the surface into 4-gons, which means $E=2F$. Therefore, $4V-2E=8$. Further, if $V_i$ is the number of vertices of degree $i$, then $V = V_2 + V_3 + V_4 + \cdots$, but also $2E = 2V_2 + 3V_3 + 4V_4 + \cdots$, and so $8 = 2V_2 + V_3 -V_5 - 2V_6 - \cdots$, which means $2V_2 + V_3 \geq 8$. In other words, this graph contains at least one vertex of degree 2 or 3, i.e., an elliptic point where 2 or 3 tiles meet. 
    
     Say there is a vertex of degree 3. There is a move to merge and separate two adjacent hyperbolic critical points of the same orientation with respect to the foliation; this was established in \cite[Lemma~3]{birman1990studying} and used as a basis in several other foliation proofs, e.g., \cite[Figure~6]{arandaXXXbooklink}. See Figure~\ref{fig:monkeysaddle} for the move and its effect on the tiling at a trivalent vertex. Note that at least two of its adjacent $h$-tiles have singularities of the same sign, and so we may apply this move to reduce the number of tiles at this elliptic point from 3 to 2.
     
     Finally, a vertex of degree 2 and its surrounding tiles correspond to a section of surface as in Figure~\ref{fig:exchange_tiles}. The hyperbolic critical points necessarily have opposite orientation (one can see this by tracing the orientation of the leaves; see, e.g., \cite[Figure~3.2]{lafountain2017braid}). These two tiles corresponds to a finger of surface looping up to pass through the binding with pieces of the booklink passing around the binding on both sides of $S$ (since all nonsingular leaves are essential). We may perform an exchange move as shown in Figure~\ref{fig:exchange}; as a result, the nonsingular leaves at the top of the picture become nonessential.
\end{proof} 

We now have sufficient tools to prove the theorem.

\begin{proof}[Proof of Theorem~\ref{thm:split}]
Let $S$ be a splitting sphere for the booklink $\lambda$. We first put it into general position by Lemma~\ref{lem:genposn}, ensuring a sufficiently well-behaved foliation which we may decompose into tiles by Lemma~\ref{lem:tiles}. We then remove SCCs by Lemma~\ref{lem:scc}. Finally, we induct on the number of elliptic points: if there are nonessential curves, we apply Lemma~\ref{lem:nonessential} to cancel a pair of elliptic points associated to one; if not, we apply Lemma~\ref{lem:exchange} to alter the surface via an exchange move and convert an essential curve to a nonessential one. Note that Lemma~\ref{lem:nonessential} created additional SCCs, but only of the type E-H, which can be eliminated via Lemma~\ref{lem:scc} without introducing new elliptic points.
\end{proof}

The primary difference between our treatment of split booklinks and Birman and Menasco's treatment of split braids lies in the resolution of the SCCs. Their method resolved E-H and H-H types more efficiently by using the fact that a braid could not cross the disk $D_{\theta}$, so $S$ could be directly surgered along it and replaced by the resulting essential sphere (note at least one of the resulting spheres must be essential). This method had the advantage of not introducing new elliptic points; however, it cannot address booklink E-H and H-H curves, and it cannot address situation E-E at all, which does not occur for braids. 

Observe also that our proof does not immediately extend to booklinks in arbitrary 3-manifolds. In particular, it relies upon the fact that any nonsingular leaf between two elliptic points cuts a page into two disks. That said, we could extend the definition of a \emph{split booklink} to an arbitrary open book 3-manifold by requiring the splitting sphere to intersect each component of the binding in two points, and then conjecture:

\begin{conj}
    If $\lambda$ is a booklink in an open book 3-manifold which represents a split link, there is a finite sequence of booklink isotopies and exchange moves which converts $\lambda$ into a split booklink.
\end{conj}

\begin{figure}
\begin{subfigure}{0.3\textwidth}\centering
\includegraphics[scale=0.25]{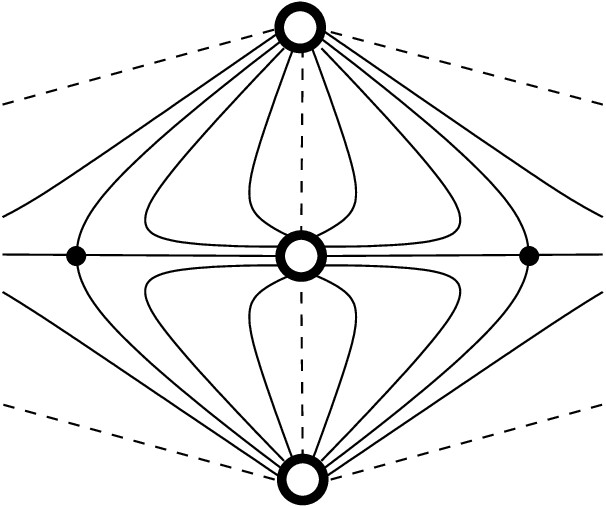}\caption{}
\end{subfigure}
\begin{subfigure}{0.3\textwidth}\centering
\includegraphics[scale=0.7]{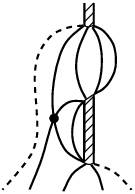}\caption{}
\end{subfigure}
\caption{The tiles involved in an exchange move.}\label{fig:exchange_tiles} 
\end{figure}

As an immediate consequence of the theorem, we obtain information about the braid and bridge indices of a split link. The bridge index is additive by definition, and Birman and Menasco provided the first proof that the braid index was also additive~\cite{birman1990studying}. Using Theorem~\ref{thm:split}, we may extend these identities to the booklink spectrum in Corollary~\ref{cor:splitspectrum}, which gives us a recipe for calculating the spectrum:
\[\bb{d}{L} = \min_{d_1+d_2=d} \big[\bb{d_1}{L_1} + \bb{d_2}{L_2}\big].\]
This formula reduces to the braid index relation at $d=0$ and contains the bridge relation because the first $d$ for which $\bbb{d} = 0$ is $\bridge(L_1) + \bridge(L_2)$.

\begin{proof}[Proof of Corollary~\ref{cor:splitspectrum}]
    The bridge index is additive by definition: if $\lambda=\lambda_1 \cup \lambda_2$ is a booklink which is a representative of the split link, then $\bridge(\lambda) = \bridge(\lambda_1) + \bridge(\lambda_2)$.

    The braid index is superadditive by definition: $\braid(\lambda) \geq \braid(\lambda_1) + \braid(\lambda_2)$. If the minimal intersection number of $\lambda_1$ may be realized on the same page as $\lambda_2$, then the braid index is additive. In fact, if $\lambda$ is a split booklink, then each $\lambda_i$ lives inside a ball pierced once by the binding, and we may freely rotate the ball around the binding to line up the pages of minimal intersection number. 
    
    Theorem~\ref{thm:split} guarantees that, if there is any booklink representative $\lambda$ of $L$ whose split components have bridge numbers $d_1$ and $d_2$, then it can be transformed into a split booklink representative $\lambda'$ without changing the bridge numbers or increasing the braid number; in fact, since it is split, the braid number is now additive.
\end{proof}

\section{Composite Links}\label{sec:composite}

Birman and Menasco demonstrated the corresponding decomposition property for braids which are connected sums in the Composite Braid Theorem \cite{birman1990studying}, with the proof corrected for an omission in \cite{birman1990studyingerratum}. 

To extend this result to booklinks, we follow the same pattern as for split links, although the removal of SCCs is slightly more complicated, and we must take care with canceling elliptic points and performing exchange moves to ensure they do not interfere with the locations where the booklink punctures the splitting sphere. 

To begin, we identify a splitting sphere $S$ for the composite booklink $\lambda$. Both Lemma~\ref{lem:genposn} on general position and Lemma~\ref{lem:tiles} on the decomposition into tiles still apply, although there will be two points where $\lambda$ punctures the tiles. We must consider the locations of these punctures and their possible impact on any other moves. The puncture may slide freely across the surface away from any singular leaves; if a puncture approaches a singular leaf, it may be able to slide it across, although it may be obstructed, as shown in Figure~\ref{fig:nudgeslidepush}.

\begin{figure}\centering
    \begin{subfigure}{0.35\textwidth}\centering
        \includegraphics{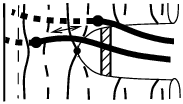} \caption{}\label{fig:nudgeslidepush_good}
    \end{subfigure}
    \begin{subfigure}{0.35\textwidth}\centering
        \includegraphics{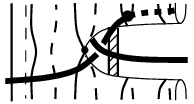} \caption{}\label{fig:nudgeslidepush_bad}
    \end{subfigure}
    \caption{\eqsubref{fig:nudgeslidepush_good} A puncture which can be isotoped across a singular leaf; \eqsubref{fig:nudgeslidepush_bad} one which cannot.}\label{fig:nudgeslidepush}
\end{figure}

\begin{lem}\label{lem:nudgeslidepush}
	Say $\lambda$ punctures a tile. We may redefine the tiling without altering the surface or booklink to move the puncture off the tile. If the puncture occurs in the product of nonsingular leaves shared by two adjacent tiles with hyperbolic singularities of opposite orientation, then there is a booklink isotopy moving the puncture across any or all of the singular leaves of at least one of the tiles. 
\end{lem}

\begin{proof}
	The puncture lives in a product neighborhood of nonsingular leaves, and we define a tiling by selecting a single leaf from this product to divide the two tiles. We may move the puncture between tiles by selecting a different leaf.

	If we wish to move a puncture across a singular leaf, we may do so as in Figure~\ref{fig:nudgeslidepush_good} if the direction of the booklink is compatible with the orientation of the tile. That is, orient the booklink to intersect the surface positively, and then compare the normal vector at the hyperbolic singularity with the tangent vector of the booklink; if they form an acute angle, then the puncture may be isotoped across the tile's singular leaf without concern. On the other hand, if the vectors form an obtuse angle, it is possible that another portion of the booklink obstructs movement across the singular leaf, as in Figure~\ref{fig:nudgeslidepush_bad}. If the hyperbolic singularity in the adjacent tile has the opposite orientation, however, then we may first slide the puncture to the adjacent tile and then move it over that tile's singular leaves. 
\end{proof}

The elimination of SCCs does not carry over directly, and we require a slight variation on Lemma~\ref{lem:scc}: 

\begin{lem}\label{lem:scc_v2}
    If a closed, connected surface $\Sigma$ is in foliated general position, then its SCCs are types E-E, E-H, or H-H. If $\Sigma$ has a SCC of type E-E, then it is a sphere without elliptic points. Any SCC of type E-H or H-H may be eliminated by a booklink isotopy on $\Sigma$ away from the booklink punctures (and, if the type E-H tile avoids punctures, this will not add any elliptic points).
\end{lem}

\begin{proof}
    Once more, we have the three cases of Figure~\ref{fig:scc}, possibly altered by the presence of punctures. The E-E and H-H cases are identical, provided we are careful to act away from the punctures when we push the surface back through the binding, adding a pair of elliptic points. In the E-H case, if the tile avoids punctures, then the move again carries through as before. If there is a puncture, though, then we may not be able to directly collapse this region of the surface to cancel out the extremal and hyperbolic critical points; however, we may divide the product neighborhood of SCCs into two pieces, a small neighborhood of the extremal singularity itself, and a cylinder containing the puncture(s) and the hyperbolic singularity. We first resolve the punctured cylinder by treating it the same way we did the H-H case and pushing it back through the binding away from the punctures, and then we may resolve the E-H case safely.
\end{proof}

Likewise, we build a composite of Lemma~\ref{lem:nonessential} and Lemma~\ref{lem:exchange} to account for the presence of punctures.

\begin{lem}\label{lem:nonessential_exchange_v2}
    Let $S$ be in foliated general position without SCCs. If there are more than 2 elliptic points, then either we may locate and cancel a pair of elliptic points connected by a family of nonessential curves away from the punctures, or else we may perform an exchange move away from the punctures followed by a cancellation of two elliptic points.
\end{lem}

\begin{proof}
    Observe that each puncture occurs in a product neighborhood of nonsingular leaves, and so there are at most two neighborhoods complicated by punctures. In any other neighborhood where we locate a nonessential leaf, we may freely apply Lemma~\ref{lem:nonessential} and cancel its terminating elliptic points.
    
    If the only remaining neighborhoods of nonsingular leaves are essential or contain punctures, then we once more consider the multigraph with a vertex for each elliptic point and an edge for each product family of nonsingular leaves. At most 4 of the elliptic points bound such a family with a puncture.
    
    As before, a bivalent vertex corresponds to a piece of the surface that looks like Figure~\ref{fig:exchange_tiles} with a pair of hyperbolic singularities of opposite sign. Lemma~\ref{lem:nudgeslidepush} allows us to push any punctures across one of the singular leaves and off the bottom of the picture, after which we may apply Lemma~\ref{lem:exchange} and then Lemma~\ref{lem:nonessential} to perform an exchange move and then cancel two elliptic points.
	
	If there are more than 2 vertices but none of degree 2, then we apply the Euler characteristic inequality of the proof of Lemma~\ref{lem:exchange}: since $2V_2 + V_3 \geq 8$, then $V_2 = 0$ implies $V_3 \geq 8$. At most 4 of the elliptic points touch regions with punctures, and so we may apply the transformation of Figure~\ref{fig:monkeysaddle} to one of the other degree-3 vertices, generate a new degree-2 vertex, and then continue with an exchange move and cancellation.
\end{proof}

    These lemmas may now be knit together in an induction to prove the theorem.
    
\begin{proof}[Proof of Theorem~\ref{thm:composite}]
    Say $S$ is a splitting sphere for a composite booklink $\lambda$. First, we put it in foliated general position by Lemma~\ref{lem:genposn} then decompose it into tiles by Lemma~\ref{lem:tiles}. Next, when we consider the location of the booklink punctures, Lemma~\ref{lem:nudgeslidepush} allows us to slide a puncture across tile boundaries by redefining the tile division or even across a singular leaf if the tile has the appropriate orientation. Lemma~\ref{lem:scc_v2} eliminates SCCs, even those which occur in the same tile as a puncture. Finally, inductive application of Lemma~\ref{lem:nonessential_exchange_v2} allows to reduce ourselves to two elliptic points by either canceling a pair connected by a nonessential leaf or by performing an exchange move to create a nonessential leaf.
\end{proof}

Once again, Theorem~\ref{thm:composite} allows a spectrum result in Corollary~\ref{cor:compositespectrum} which mirrors the classical bridge and braid formulas. At $d=0$, it restricts to Birman and Menasco's result in~\cite{birman1990studying} that the braid index is additive (minus one) under composites. Similarly, $\bb{d}{L} = 0$ exactly for $d \geq \bridge(L_1) + \bridge(L_2)-1$, matching Schubert's result in~\cite{schubert1956knoten} that the bridge index is also additive (minus one). The corollary additionally provides a formula for the intermediate values $0 < d < \bridge(L)$:
\[\bb{d}{L} = \min_{d_1+d_2=d+1}\big[\bb{d_1}{L_1} + \bb{d_2}{L_2}\big].\]

\begin{figure}\centering
    \begin{subfigure}{\textwidth}\centering
        {\pgfkeys{/pgf/fpu/.try=false}%
\ifx\XFigwidth\undefined\dimen1=0pt\else\dimen1\XFigwidth\fi
\divide\dimen1 by 5982
\ifx\XFigheight\undefined\dimen3=0pt\else\dimen3\XFigheight\fi
\divide\dimen3 by 1076
\ifdim\dimen1=0pt\ifdim\dimen3=0pt\dimen1=3946sp\dimen3\dimen1
  \else\dimen1\dimen3\fi\else\ifdim\dimen3=0pt\dimen3\dimen1\fi\fi
\tikzpicture[x=+\dimen1, y=+\dimen3]
{\ifx\XFigu\undefined\catcode`\@11
\def\temp{\alloc@1\dimen\dimendef\insc@unt}\temp\XFigu\catcode`\@12\fi}
\XFigu3946sp
\ifdim\XFigu<0pt\XFigu-\XFigu\fi
\clip(408,-1963) rectangle (6390,-887);
\tikzset{inner sep=+0pt, outer sep=+0pt}
\pgfsetfillcolor{black}
\pgftext[base,left,at=\pgfqpointxy{6161}{-1776}] {\fontsize{8}{9.6}\usefont{T1}{ptm}{m}{n}$\tau_2$}
\pgfsetlinewidth{+7.5\XFigu}
\pgfsetstrokecolor{black}
\draw (2250,-1464)--(2550,-1464);
\draw (2250,-1535)--(2550,-1535);
\draw (4575,-1425)--(4725,-1500)--(4575,-1575);
\draw (4350,-1464)--(4650,-1464);
\draw (4350,-1535)--(4650,-1535);
\draw  (1270,-1500) circle [radius=+34];
\draw (675,-1500)--(1235,-1501);
\draw  (5545,-1500) circle [radius=+34];
\draw (4950,-1500)--(5510,-1501);
\draw  (3445,-1500) circle [radius=+34];
\draw (2850,-1500)--(3410,-1501);
\pgfsetroundcap
\pgfsetdash{{+60\XFigu}{+60\XFigu}}{++0pt}
\draw  (3450,-1500) circle [radius=+450];
\draw  (5550,-1500) circle [radius=+450];
\pgfsetlinewidth{+15\XFigu}
\pgfsetdash{}{+0pt}
\pgfsetstrokecolor{white}
\draw  (6300,-1575) circle [radius=+75];
\pgfsetcolor{black}
\filldraw  (5909,-1516) circle [radius=+22];
\filldraw  (6095,-1486) circle [radius=+22];
\pgfsetlinewidth{+7.5\XFigu}
\pgfsetdash{{+60\XFigu}{+60\XFigu}}{++0pt}
\draw  (1276,-1506) circle [radius=+450];
\pgfsetbuttcap
\pgfsetlinewidth{+15\XFigu}
\pgfsetdash{}{+0pt}
\draw (3695,-1435)--(3645,-1315)--(3750,-1385);
\draw (4035,-1590)--(4140,-1670)--(4095,-1540);
\draw (1538,-1463)--(1488,-1343)--(1593,-1413);
\draw (1845,-1570)--(1950,-1650)--(1905,-1520);
\draw (5831,-1486)--(5786,-1357)--(5882,-1447);
\draw (6110,-1543)--(6212,-1639)--(6167,-1498);
\draw (3675,-1350)--(3825,-1575);
\draw (3975,-1425)--(4125,-1650);
\draw (3750,-1275)--(3825,-1390);
\draw (3960,-1605)--(4050,-1725);
\pgfsetbeveljoin
\draw (1500,-1350)--(1500,-1351)--(1503,-1355)--(1510,-1364)--(1520,-1379)--(1534,-1398)
  --(1549,-1420)--(1565,-1442)--(1581,-1462)--(1595,-1480)--(1608,-1496)--(1619,-1509)
  --(1630,-1520)--(1640,-1530)--(1650,-1538)--(1662,-1545)--(1673,-1552)--(1685,-1558)
  --(1698,-1564)--(1711,-1569)--(1723,-1573)--(1736,-1577)--(1748,-1582)--(1759,-1586)
  --(1769,-1590)--(1779,-1595)--(1788,-1600)--(1796,-1606)--(1804,-1614)--(1812,-1623)
  --(1821,-1634)--(1830,-1648)--(1840,-1664)--(1850,-1681)--(1860,-1698)--(1868,-1711)
  --(1873,-1720)--(1875,-1724)--(1875,-1725);
\draw (1575,-1275)--(1575,-1276)--(1577,-1280)--(1582,-1289)--(1590,-1302)--(1600,-1319)
  --(1610,-1336)--(1620,-1352)--(1629,-1366)--(1638,-1377)--(1646,-1386)--(1654,-1394)
  --(1663,-1400)--(1671,-1405)--(1681,-1410)--(1691,-1414)--(1702,-1418)--(1714,-1423)
  --(1727,-1427)--(1739,-1431)--(1752,-1436)--(1765,-1442)--(1777,-1448)--(1788,-1455)
  --(1800,-1463)--(1810,-1470)--(1820,-1480)--(1831,-1491)--(1842,-1504)--(1855,-1520)
  --(1869,-1538)--(1885,-1558)--(1901,-1580)--(1916,-1602)--(1930,-1621)--(1940,-1636)
  --(1947,-1645)--(1950,-1649)--(1950,-1650);
\draw (5775,-1350)--(5775,-1351)--(5778,-1354)--(5784,-1363)--(5794,-1378)--(5807,-1396)
  --(5821,-1417)--(5836,-1437)--(5850,-1456)--(5862,-1472)--(5872,-1486)--(5881,-1496)
  --(5889,-1504)--(5895,-1509)--(5900,-1513)--(5905,-1514)--(5909,-1514)--(5912,-1513)
  --(5915,-1509)--(5917,-1505)--(5919,-1499)--(5920,-1491)--(5920,-1483)--(5920,-1474)
  --(5919,-1464)--(5918,-1455)--(5917,-1445)--(5915,-1435)--(5913,-1425)--(5909,-1413)
  --(5906,-1402)--(5901,-1389)--(5895,-1374)--(5888,-1358)--(5881,-1340)--(5872,-1322)
  --(5864,-1304)--(5857,-1289)--(5852,-1280)--(5850,-1276)--(5850,-1275);
\draw (6225,-1650)--(6225,-1649)--(6222,-1646)--(6216,-1637)--(6206,-1622)--(6193,-1604)
  --(6179,-1583)--(6164,-1563)--(6150,-1544)--(6138,-1528)--(6128,-1514)--(6119,-1504)
  --(6111,-1496)--(6105,-1491)--(6100,-1488)--(6095,-1486)--(6091,-1486)--(6088,-1487)
  --(6085,-1491)--(6083,-1495)--(6081,-1501)--(6080,-1509)--(6080,-1517)--(6080,-1526)
  --(6081,-1536)--(6082,-1545)--(6083,-1555)--(6085,-1565)--(6088,-1575)--(6091,-1587)
  --(6094,-1598)--(6099,-1611)--(6105,-1626)--(6112,-1642)--(6119,-1660)--(6128,-1678)
  --(6136,-1696)--(6143,-1711)--(6148,-1720)--(6150,-1724)--(6150,-1725);
\pgftext[base,left,at=\pgfqpointxy{1201}{-990}] {\fontsize{8}{9.6}\usefont{T1}{ptm}{m}{n}$S$}
\pgftext[base,left,at=\pgfqpointxy{423}{-1402}] {\fontsize{8}{9.6}\usefont{T1}{ptm}{m}{n}$H_{min}$}
\pgftext[base,left,at=\pgfqpointxy{1434}{-1298}] {\fontsize{8}{9.6}\usefont{T1}{ptm}{m}{n}$\tau_1$}
\pgftext[base,left,at=\pgfqpointxy{1886}{-1779}] {\fontsize{8}{9.6}\usefont{T1}{ptm}{m}{n}$\tau_2$}
\pgftext[base,left,at=\pgfqpointxy{3608}{-1264}] {\fontsize{8}{9.6}\usefont{T1}{ptm}{m}{n}$\tau_1$}
\pgftext[base,left,at=\pgfqpointxy{4063}{-1799}] {\fontsize{8}{9.6}\usefont{T1}{ptm}{m}{n}$\tau_2$}
\pgftext[base,left,at=\pgfqpointxy{5712}{-1307}] {\fontsize{8}{9.6}\usefont{T1}{ptm}{m}{n}$\tau_1$}
\pgfsetmiterjoin
\pgfsetlinewidth{+7.5\XFigu}
\draw (2475,-1425)--(2625,-1500)--(2475,-1575);
\endtikzpicture}%
        \caption{}\label{fig:sum_opposite}
    \end{subfigure}\\\medskip
    \begin{subfigure}{\textwidth}\centering
        \input{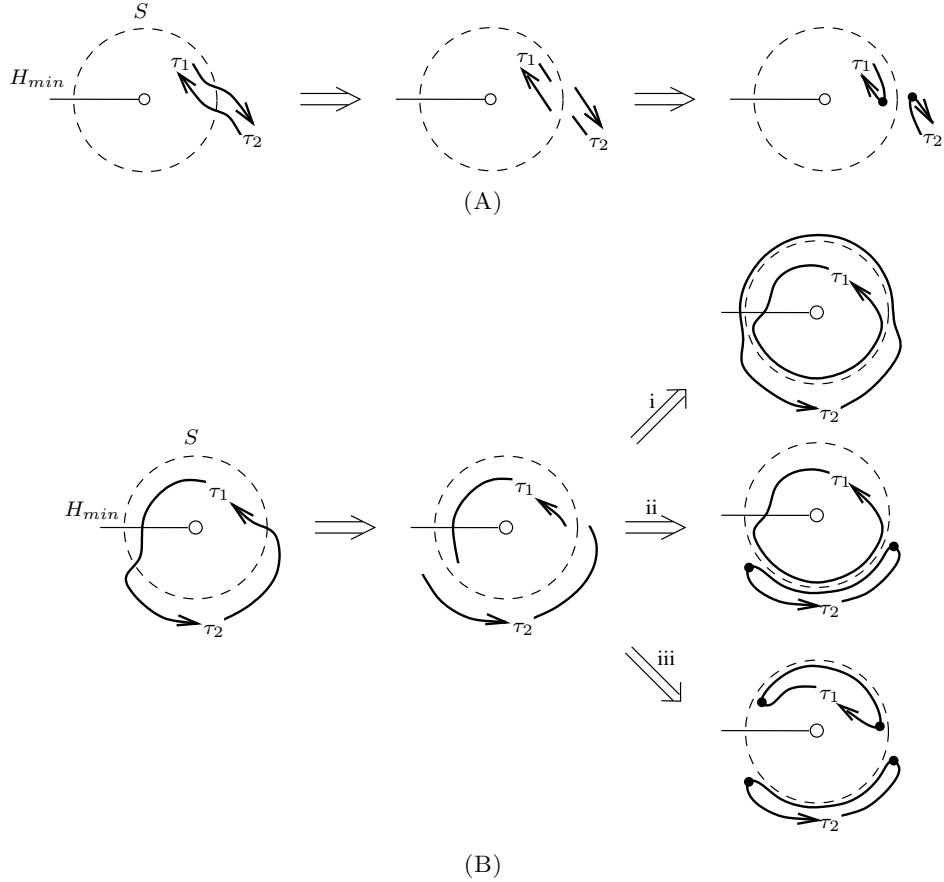}
        \caption{}\label{fig:sum_same}
    \end{subfigure}
    \caption{Decomposing a composite booklink into booklink summands when the splitting sphere $S$ intersects it at points of \eqsubref{fig:sum_opposite} the opposite orientation or \eqsubref{fig:sum_same} the same orientation.}\label{fig:sum}
\end{figure}

This corollary is less straightforward than that for split links because there are several possible distinct ways to decompose a booklink into two summands, as shown in Figure~\ref{fig:sum}. We can see from the figure that the spectrum is at worst additive. To prove the exact formula, we will evaluate the bounds given by each decomposition and note that the one given in Figure~\ref{fig:sum_opposite} is at least as good as the one in~\ref{fig:sum_same}, and any booklink other than a pure braid may be constructed by this method.

\begin{proof}[Proof of Corollary~\ref{cor:compositespectrum}]
    Say that $\lambda$ represents a composite link $L_1 \# L_2$. By Theorem~\ref{thm:composite}, we may transform the booklink into a composite booklink by a series of booklink isotopies and exchange moves, after which the splitting sphere $S$ is in foliated general position. We cut $\lambda$ at $S$ into two tangles $\tau_1$ and $\tau_2$ which we close off into booklinks $\lambda_1$ and $\lambda_2$. We use $d$ (respectively, $d_i$) and $n$ (respectively, $n_i$) for the bridge and braid indices of $\lambda$ (respectively, $\lambda_i$).

    We consider two cases separately as shown in Figure~\ref{fig:sum}. By general position, $\lambda$ does not have a critical point where it pierces $S$, and so it is locally wrapping around the binding either clockwise or counterclockwise. The situation naturally breaks into two cases depending on whether it is wrapping in the same direction at both punctures or not.  
    
    As in Figure~\ref{fig:sum_opposite}, if $\lambda$ is oriented in opposite directions at the points where it pierces $S$, we may decompose it into tangles $\tau_i$ and then close them off into booklinks $\lambda_i$, thereby increasing the bridge index by one without altering the braid index, i.e., $d = d_1 + d_2 - 1$ and $n = n_1 + n_2$. Each arc has an odd number of critical points, so we cannot close it off without adding one more. Additionally, we may assume that the $\lambda_i$ attains the minimal possible braid index for the given number of critical points and that it does so on the same page, which we call $H_{min}$. If not, we may replace $\lambda_i$ by a different booklink representative $\lambda_i'$ which does attain the minimum, then we may rotate one so that they both have minimal intersection number with the same page $H_{min}$, and then we may isotope $\lambda_i'$ to move one critical point over to $S$ at the appropriate location without passing through $H_{min}$; this may change the booklink types but not the link types or the indices $(d_i,n_i)$. Therefore, we may conclude, 
    \begin{equation}
    d = d_1 + d_2 - 1 \qquad \text{and} \qquad n = \bb{d_1}{L_1} + \bb{d_2}{L_2}.\label{eqn:sum_opposite}
    \end{equation}

    As in Figure~\ref{fig:sum_same}, if $\lambda$ has the same orientation at both points where it meets $S$, then there are two ways to close off each tangle $\tau_i$ into a booklink $\lambda_i$, either by a move that preserves orientation or one that reverses it twice: (i) both are closed by the orientation-preserving move; (ii) has one of each; (iii) both are closed by the move that does not preserve orientation. The orientation-preserving move does not alter the number of critical points, whereas the other move adds a pair. The effect on braid index is more subtle. Observe that the move which is non-orientation-preserving, for example $\lambda_2$ in (iii), may not attain the minimal braid index in its booklink class because it is not generic to have a diagram with an arc between two critical points without any crossings. For example, see Figure~\ref{fig:fig8-11}; if there were a $(1,1)$-representative for the fig. 8 knot with such an arc, then there would be a $(0,2)$-representative for it as well. On the other hand, the orientation-preserving move produces a generic booklink such as $\lambda_1$ in (i) that we may assume minimizes braid index at any fixed page $H_{min}$. If not, we can replace it with a booklink that realizes the minimum intersection number with the desired page. We can also scale the $\theta$ so that $\lambda$ in fact has no crossings or critical points in the bottom half of its disk, so the arc depicted is present.

    Therefore, (i) gives us the bounds of 
    \begin{equation}
    d = d_1 + d_2 \qquad \text{and} \qquad  n = \bb{d_1}{L_1} + \bb{d_2}{L_2} - 1.\label{eqn:sum_same}
    \end{equation}
    We will later minimize over all choices of $d_1+d_2=d$, and we now show that methods (ii) and (iii) can do no better when we pass to the minimum. Method (ii) satisfies $n=n_1+n_2$, where $n_1=\bb{d_1}{L_1}$ and $n_2 \geq \bb{d_2}{L_2}$, and the inequality may be strict; however, by the construction, we can assume $n_2 =\bb{d_2-1}{L_2}$, in which case \[d = d_1 + (d_2-1) \qquad \text{and} \qquad  n = \bb{d_1}{L_1} + \bb{d_2-1}{L_2} - 1,\] which is subsumed by Equation~\ref{eqn:sum_same} when we pass to the minimum. Similarly, for (iii), $n=n_1+n_2+1$ where we may assume $n_i=\bb{d_i-1}{L_i}-1$, and so \[d=(d_1-1)+(d_2-1) \qquad \text{and} \qquad n = \bb{d_1-1}{L_1} + \bb{d_2-1}{L_2} - 1.\]
    
    Finally, we establish the bounds. If $d=0$, then $\lambda$ is a pure braid and necessarily meets the splitting sphere at two points where it has the same orientation, and so we use Equation~\ref{eqn:sum_same} with $d=0$, \[\bb{0}{L} = \bb{0}{L_1} + \bb{0}{L_2}.\] On the other hand, if $d \geq \bridge(L_1) + \bridge(L_2)$, i.e., $d > \bridge(L)$, then we may realize $\lambda$ using two plats summed at critical points as in Figure~\ref{fig:sum_opposite}, and \[\bb{d}{L} = 0.\]
    
    For intermediate $d$, we must minimize $n$ over all booklinks with spheres $S$ of whichever type can occur. Observe first that any splitting sphere with punctures of the same orientation as in Figure~\ref{fig:sum_same} may be converted into one with opposite oriented punctures as in Figure~\ref{fig:sum_opposite}: we can select a small region around one of the punctures and push a finger of the sphere along $\lambda$ until it passes through a critical point, changing the orientation of the puncture; with sufficient care, this finger will not intersect the binding. Next, if a given booklink admits splitting spheres of both possible types, then we may restrict our attention to the type in Figure~\ref{fig:sum_opposite}: $\bb{d_i+1}{L_i}\leq\bb{d_i}{L_i}-1$ as long as $d_i < \bridge(L_i)$, so
    \[\min_{d = d_1 + d_2 - 1} \big[\bb{d_1}{L_1} + \bb{d_2}{L_2}\big] \leq \min_{d = d_1 + d_2} \big[\bb{d_1}{L_1} + \bb{d_2}{L_2}-1\big].\] 
    Therefore, if $0 < d < \bridge(L_1)+\bridge(L_2)$,
    \[\bb{d}{L} = \min_{d_1+d_2=d+1}\big[\bb{d_1}{L_1} + \bb{d_2}{L_2}\big].\]
\end{proof}

\section{Booklink Spectra}\label{sec:spectra}

\begin{figure}\centering
    \includegraphics[scale=0.5]{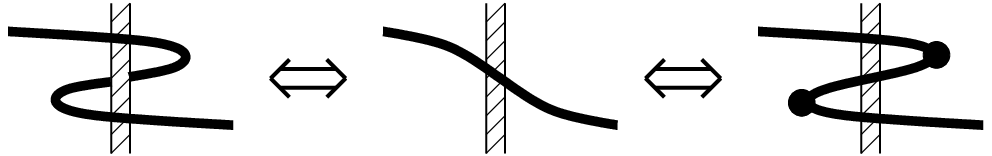}
    \caption{Two operations which increase bridge and braid index.}\label{fig:stable}
\end{figure}

With these recipes for calculating the spectra of split and composite booklinks, we may restrict our attention to non-split, prime links. In this paper, we calculate the spectra for knots up through 9 crossings.

First, we establish the general behavior of the spectrum. If a booklink has a $(d,n)$-representative, we may always generate a $(d+1,n)$-representative by perturbing locally to add a pair of critical points, or we may generate a $(d,n+1)$-representative by adding a braid-style stabilization; see Figure~\ref{fig:stable}. It is sometimes possible to trade braid index for bridge index, and vice versa. First, if braid index is positive, we may always trade an intersection with $H_{min}$ for a pair of critical points, which tells us the invariant $\bbb{d}$ is monotonic.

\begin{lem}\label{lem:monotonic}
    As a function of $d$, $\bbb{d}$ is strictly decreasing until it reaches zero.
\end{lem}

Roughly speaking, we convert the left booklink in Figure~\ref{fig:stable} into the right booklink.

\begin{proof}
    Say $\bbb{d} > 0$. Select a $(d,\bbb{d})$-representative for the booklink. Pick any intersection with $H_{min}$ and find a path to the binding through $H_{min}$. Push the link along the path to eliminate an intersection with $H_{min}$ and add a pair of critical points. This is now a $(d+1,\bbb{d} - 1)$-representative, so $\bbb{d+1} \leq \bbb{d} - 1$. 
\end{proof}

Similarly, while it is usually not possible to convert a pair of critical points into an intersection with $H_{min}$, it is possible if the booklink is in classical plat form, which determines the end of the $\{\bbb{d}\}$ sequence.

\begin{prop}\label{prop:endings}
    If a link has braid index $\mathfrak{n}$ and bridge index $\mathfrak{d}$, then $\bbb{0} = \mathfrak{n}$, $\bbb{\mathfrak{d}} = 0$ and $\bbb{\mathfrak{d}-1} = 1$; in particular, its spectrum is $\{\mathfrak{n}, \ldots, 1, 0\}$.
\end{prop}

\begin{proof}
    By definition, $\bbb{\mathfrak{d}} = 0$ and $\bbb{0} = \mathfrak{n}$. Additionally, the link has a $2\mathfrak{d}$-plat presentation where one of the plat strands is not involved in any crossings, i.e., a booklink representative with bridge index $\mathfrak{d}$ and braid index $0$ with one arc between a pair of critical points with unobstructed access to the binding. Push this strand through the binding, increasing the braid index to $1$ and decreasing the bridge index to $\mathfrak{d}-1$.
\end{proof}

In particular, the spectrum for a 2-bridge link is determined by its braid index.

\begin{cor}\label{cor:2-bridge}
    Any 2-bridge link has spectrum $(\mathfrak{n},1,0)$.
\end{cor}

Likewise, the spectrum is forced for the class of links called \emph{BB-links} whose bridge and braid indices are the same.

\begin{cor}\label{cor:bb}
    A BB-link with classical bridge and braid indices of $\mathfrak{d}$ has spectrum $\{\mathfrak{d}, \mathfrak{d}-1, \cdots, 1, 0\}$.
\end{cor}


We now determine the spectra for all knots through 9 crossings. The above results provide enough information to determine the spectra for all knots through 8 crossings (except $8_{15}$) and 25 of the 49 knots with 9 crossings. The remaining 8- and 9-crossing knots are all 3-bridge and attain the lowest possible spectrum allowed, so we complete the tabulation of the spectra by demonstrating a booklink in Figure~\ref{fig:missing} which realizes the only unknown value in each spectrum. We list the results in Table~\ref{table:knot_spectra} and state that they are complete in Theorem~\ref{thm:knot_spectra}.

\begin{proof}[Proof of Theorem~\ref{thm:knot_spectra}]
The bridge and braid indices found in the table are taken from KnotInfo \cite{knotinfo}. The 2-bridge knots are determined by Corollary~\ref{cor:2-bridge} and the BB-links by Corollary~\ref{cor:bb}. The remaining knots are all 3-bridge, and so the spectrum is necessarily $\{\mathfrak{n}, \bbb{1}, 1, 0\}$, with $1 < \bbb{1} < \mathfrak{n}$ by Lemma~\ref{lem:monotonic}, and we demonstrate that $\bbb{1} = 2$ by providing a $(1,2)$-representative for each booklink in Figure~\ref{fig:missing}. These diagrams may be examined visually to verify that they have the required invariants, and we have calculated the DT-notation and checked it against KnotInfo to verify that these booklinks represent the appropriate links.
\end{proof}

\section{Future Work}\label{sec:futurework}

We recall our original motivation, to describe mathematically the process of applying Alexander's Theorem to a given knot or link. Experimentation suggests several interesting patterns; one of which is when we reach a point partway through we realize the next arc can only be resolved by two stabilizations, not by one, then any further arc will also require at least two. This suggests that the resolution of critical points gets progressively more difficult, rather than easier. We can express this in terms of the spectrum:

\begin{conj}\label{conj:concavity}
    The sequence $\{\bbb{d}\}$ is concave upwards, that is, \[\bbb{d-1} - \bbb{d} \geq \bbb{d} - \bbb{d+1}.\]
\end{conj}

We may also profit from more explicitly describing this process of applying Alexander's Theorem. It may be enlightening to be able to describe explicitly how to proceed from one booklink to another. We are inspired by the idea of a stabilization-free version of Markov's Theorem similar to the one for braids.

\begin{ques}
    Given two $(d,n)$-representatives of booklinks representing the same link, what is the minimum set of moves required to pass from one to another? Given booklinks with different braid and/or bridge indices, what is the minimum set of moves required?
\end{ques}

A more immediate question is suggested by the simplicity of the knots listed in Table~\ref{table:knot_spectra}: All of them attain the minimum spectra allowed by the results above; however, this is not true in general for links, as their spectra may contain more information than just the classical bridge and braid indices.

\begin{figure}
\begin{subfigure}{0.45\textwidth}\centering
	\includegraphics[scale=0.5]{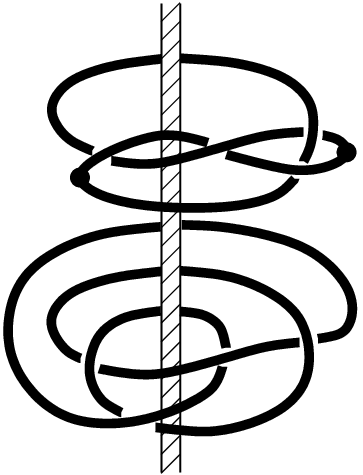}
	\end{subfigure}
\begin{subfigure}{0.45\textwidth}\centering
	\includegraphics[scale=0.5]{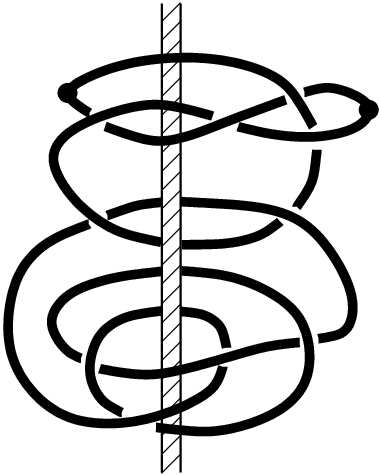}
	\end{subfigure}
    \caption{Two links built from $4_1$ with spectrum $\{6,4,2,1,0\}$.}\label{fig:notboring}
\end{figure}

\begin{ex}\label{ex:notboring}
See Figure~\ref{fig:notboring}. Corollary~\ref{cor:splitspectrum} tells us the split link with two copies of the figure 8 has spectrum $\{6,4,2,1,0\}$. Similarly, two singly-linked copies of the figure 8 clearly can achieve the same spectrum by Corollary~\ref{cor:splitspectrum} and the above $(1,4)$-representative. 
\end{ex}

There is no reason to believe that knots with similar spectra do not also exist; however, they appear to occur at much higher crossing number. In fact, we see no reason not to suppose that Lemma~\ref{lem:monotonic}, Proposition~\ref{prop:endings}, and Conjecture~\ref{conj:concavity} are both necessary and sufficient conditions for a sequence to be the spectrum of some knot.

\begin{conj}
    For any sequence of decreasing integers terminating in $\{\cdots, 1, 0\}$, there is a knot with this sequence as its bridge-braid spectrum.
\end{conj}

We could likely extend this to determine the number of components:

\begin{conj}
    Any sequence $\{\mathfrak{n}, \cdots, 1, 0\}$ satisfying a set of reasonable conditions is the spectrum for an $n$-component link with this spectrum.
\end{conj}

\begin{table}
\begin{tabular}{lllllllllll}
\cmidrule[\heavyrulewidth]{1-5}\cmidrule[\heavyrulewidth]{7-11}
& $\mathfrak{d}$ & $\mathfrak{n}$ & Spectrum & Reference &&  & $\mathfrak{d}$ & $\mathfrak{n}$ & Spectrum & Reference\\
\cmidrule{1-5}\cmidrule{7-11}
$0_1$ & 1 & 1 & $\{1,0\}$ & BB & \qquad\qquad & $9_8$ & 2 & 5 & $\{5,1,0\}$ & 2-bridge \\
$3_1$ & 2 & 2 & $\{2,1,0\}$ & 2-bridge & & $9_9$ & 2 & 3 & $\{3,1,0\}$ & 2-bridge \\
$4_1$ & 2 & 3 & $\{3,1,0\}$ & 2-bridge & & $9_{10}$ & 2 & 4 & $\{4,1,0\}$ & 2-bridge \\
$5_1$ & 2 & 2 & $\{2,1,0\}$ & 2-bridge & & $9_{11}$ & 2 & 4 & $\{4,1,0\}$ & 2-bridge \\
$5_2$ & 2 & 3 & $\{3,1,0\}$ & 2-bridge & & $9_{12}$ & 2 & 5 & $\{5,1,0\}$ & 2-bridge \\
$6_1$ & 2 & 4 & $\{4,1,0\}$ & 2-bridge & & $9_{13}$ & 2 & 4 & $\{4,1,0\}$ & 2-bridge \\
$6_2$ & 2 & 3 & $\{3,1,0\}$ & 2-bridge & & $9_{14}$ & 2 & 5 & $\{5,1,0\}$ & 2-bridge \\
$6_3$ & 2 & 3 & $\{3,1,0\}$ & 2-bridge & & $9_{15}$ & 2 & 5 & $\{5,1,0\}$ & 2-bridge \\
$7_1$ & 2 & 2 & $\{2,1,0\}$ & 2-bridge & & $9_{16}$ & 3 & 3 & $\{3,2,1,0\}$ & BB \\
$7_2$ & 2 & 4 & $\{4,1,0\}$ & 2-bridge & & $9_{17}$ & 2 & 4 & $\{4,1,0\}$ & 2-bridge \\
$7_3$ & 2 & 3 & $\{3,1,0\}$ & 2-bridge & & $9_{18}$ & 2 & 4 & $\{4,1,0\}$ & 2-bridge \\
$7_4$ & 2 & 4 & $\{4,1,0\}$ & 2-bridge & & $9_{19}$ & 2 & 5 & $\{5,1,0\}$ & 2-bridge \\
$7_5$ & 2 & 3 & $\{3,1,0\}$ & 2-bridge & & $9_{20}$ & 2 & 4 & $\{4,1,0\}$ & 2-bridge \\
$7_6$ & 2 & 4 & $\{4,1,0\}$ & 2-bridge & & $9_{21}$ & 2 & 5 & $\{5,1,0\}$ & 2-bridge \\
$7_7$ & 2 & 4 & $\{4,1,0\}$ & 2-bridge & & $9_{22}$ & 3 & 4 & $\{4,2,1,0\}$ & Fig.~\ref{fig:missing} \\
$8_1$ & 2 & 5 & $\{5,1,0\}$ & 2-bridge & & $9_{23}$ & 2 & 4 & $\{4,1,0\}$ & 2-bridge \\
$8_2$ & 2 & 3 & $\{3,1,0\}$ & 2-bridge & & $9_{24}$ & 3 & 4 & $\{4,2,1,0\}$ & Fig.~\ref{fig:missing} \\
$8_3$ & 2 & 5 & $\{5,1,0\}$ & 2-bridge & & $9_{25}$ & 3 & 5 & $\{5,2,1,0\}$ & Fig.~\ref{fig:missing} \\
$8_4$ & 2 & 4 & $\{4,1,0\}$ & 2-bridge & & $9_{26}$ & 2 & 4 & $\{4,1,0\}$ & 2-bridge \\
$8_5$ & 3 & 3 & $\{3,2,1,0\}$ & BB & & $9_{27}$ & 2 & 4 & $\{4,1,0\}$ & 2-bridge \\
$8_6$ & 2 & 4 & $\{4,1,0\}$ & 2-bridge & & $9_{28}$ & 3 & 4 & $\{4,2,1,0\}$ & Fig.~\ref{fig:missing} \\
$8_7$ & 2 & 3 & $\{3,1,0\}$ & 2-bridge & & $9_{29}$ & 3 & 4 & $\{4,2,1,0\}$ & Fig.~\ref{fig:missing} \\
$8_8$ & 2 & 4 & $\{4,1,0\}$ & 2-bridge & & $9_{30}$ & 3 & 4 & $\{4,2,1,0\}$ & Fig.~\ref{fig:missing} \\
$8_9$ & 2 & 3 & $\{3,1,0\}$ & 2-bridge & & $9_{31}$ & 2 & 4 & $\{4,1,0\}$ & 2-bridge \\
$8_{10}$ & 3 & 3 & $\{3,2,1,0\}$ & BB & & $9_{32}$ & 3 & 4 & $\{4,2,1,0\}$ & Fig.~\ref{fig:missing} \\
$8_{11}$ & 2 & 4 & $\{4,1,0\}$ & 2-bridge & & $9_{33}$ & 3 & 4 & $\{4,2,1,0\}$ & Fig.~\ref{fig:missing} \\
$8_{12}$ & 2 & 5 & $\{5,1,0\}$ & 2-bridge & & $9_{34}$ & 3 & 4 & $\{4,2,1,0\}$ & Fig.~\ref{fig:missing} \\
$8_{13}$ & 2 & 4 & $\{4,1,0\}$ & 2-bridge & & $9_{35}$ & 3 & 5 & $\{5,2,1,0\}$ & Fig.~\ref{fig:missing} \\
$8_{14}$ & 2 & 4 & $\{4,1,0\}$ & 2-bridge & & $9_{36}$ & 3 & 4 & $\{4,2,1,0\}$ & Fig.~\ref{fig:missing} \\
$8_{15}$ & 3 & 4 & $\{4,2,1,0\}$ & Fig.~\ref{fig:missing} & & $9_{37}$ & 3 & 5 & $\{5,2,1,0\}$ & Fig.~\ref{fig:missing} \\
$8_{16}$ & 3 & 3 & $\{3,2,1,0\}$ & BB & & $9_{38}$ & 3 & 4 & $\{4,2,1,0\}$ & Fig.~\ref{fig:missing} \\
$8_{17}$ & 3 & 3 & $\{3,2,1,0\}$ & BB & & $9_{39}$ & 3 & 5 & $\{5,2,1,0\}$ & Fig.~\ref{fig:missing} \\
$8_{18}$ & 3 & 3 & $\{3,2,1,0\}$ & BB & & $9_{40}$ & 3 & 4 & $\{4,2,1,0\}$ & Fig.~\ref{fig:missing} \\
$8_{19}$ & 3 & 3 & $\{3,2,1,0\}$ & BB & & $9_{41}$ & 3 & 5 & $\{5,2,1,0\}$ & Fig.~\ref{fig:missing} \\
$8_{20}$ & 3 & 3 & $\{3,2,1,0\}$ & BB & & $9_{42}$ & 3 & 4 & $\{4,2,1,0\}$ & Fig.~\ref{fig:missing} \\
$8_{21}$ & 3 & 3 & $\{3,2,1,0\}$ & BB & & $9_{43}$ & 3 & 4 & $\{4,2,1,0\}$ & Fig.~\ref{fig:missing} \\
$9_1$ & 2 & 2 & $\{2,1,0\}$ & 2-bridge & & $9_{44}$ & 3 & 4 & $\{4,2,1,0\}$ & Fig.~\ref{fig:missing} \\
$9_2$ & 2 & 5 & $\{5,1,0\}$ & 2-bridge & & $9_{45}$ & 3 & 4 & $\{4,2,1,0\}$ & Fig.~\ref{fig:missing} \\
$9_3$ & 2 & 3 & $\{3,1,0\}$ & 2-bridge & & $9_{46}$ & 3 & 4 & $\{4,2,1,0\}$ & Fig.~\ref{fig:missing} \\
$9_4$ & 2 & 4 & $\{4,1,0\}$ & 2-bridge & & $9_{47}$ & 3 & 4 & $\{4,2,1,0\}$ & Fig.~\ref{fig:missing} \\
$9_5$ & 2 & 5 & $\{5,1,0\}$ & 2-bridge & & $9_{48}$ & 3 & 4 & $\{4,2,1,0\}$ & Fig.~\ref{fig:missing} \\
$9_6$ & 2 & 3 & $\{3,1,0\}$ & 2-bridge & & $9_{49}$ & 3 & 4 & $\{4,2,1,0\}$ & Fig.~\ref{fig:missing} \\
$9_7$ & 2 & 4 & $\{4,1,0\}$ & 2-bridge\\
\cmidrule[\heavyrulewidth]{1-5}\cmidrule[\heavyrulewidth]{7-11}
\end{tabular}
\caption{Knot spectra, with classical bridge $\mathfrak{d}$ and braid $\mathfrak{n}$ indices.}\label{table:knot_spectra}
\end{table}

\begin{figure}[b]
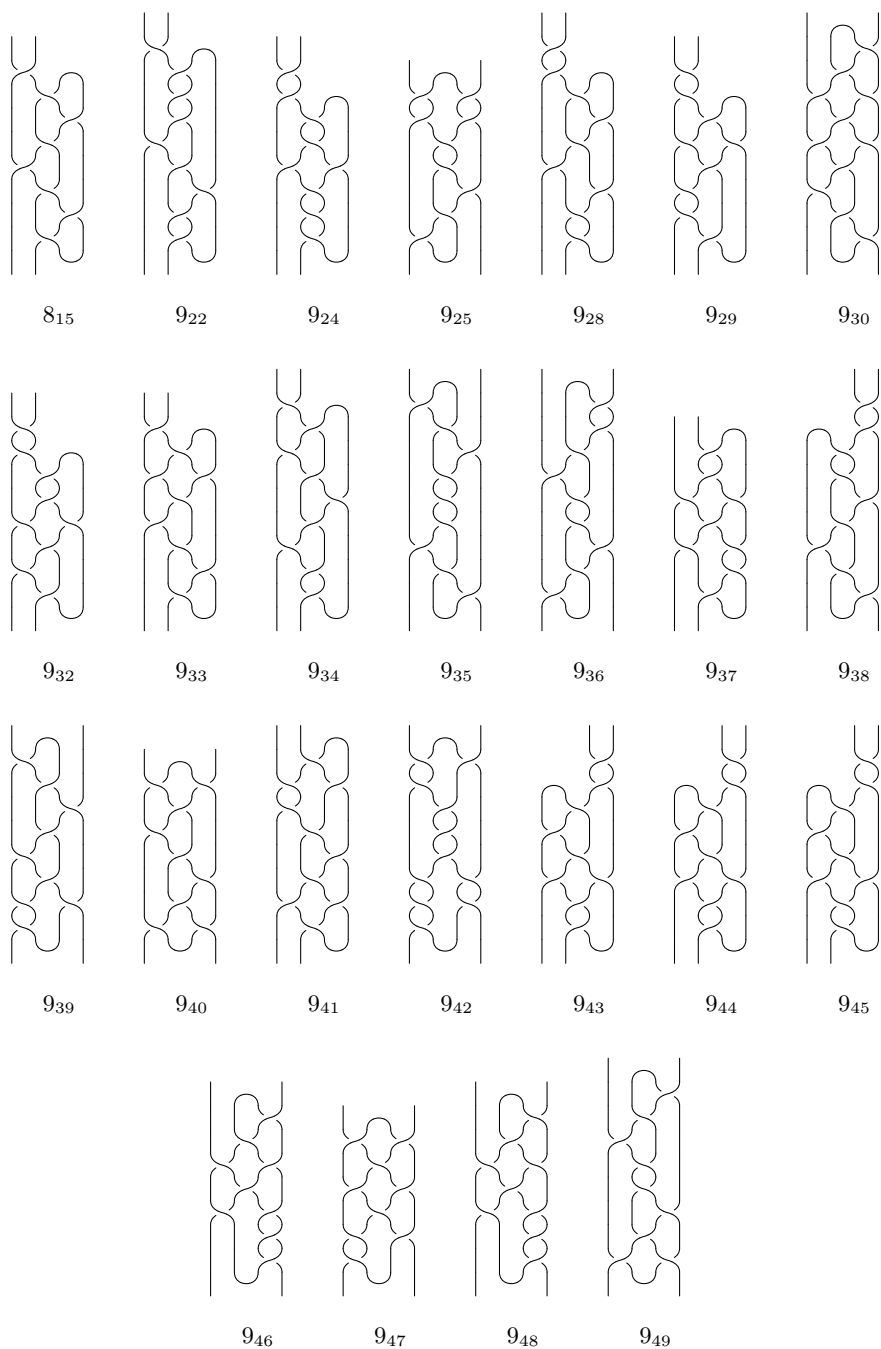
\centering
\begin{subfigure}{0.13\textwidth}\[\xygraph{ !{/r0.75pc/:}
    !{\xcapv@(0)} [ur]!{\xcapv@(0)}
    [l]!{\vtwist}[r][r]!{\vcap}[r]
    [lll]!{\xcapv@(0)} [ur]!{\vcross}[r][ur]!{\xcapv@(0)}
    [lll]!{\xcapv@(0)} [ur]!{\xcapv@(0)}[ur]!{\vtwist}[r]
    [lll]!{\xcapv@(0)} [ur]!{\vcross}[r][ur]!{\xcapv@(0)}
    [lll]!{\vtwist}[r][ur]!{\xcapv@(0)}[ur]!{\xcapv@(0)}
    [lll]!{\xcapv@(0)} [ur]!{\vcross}[r][ur]!{\xcapv@(0)}
    [lll]!{\xcapv@(0)} [ur]!{\xcapv@(0)}[ur]!{\vtwist}[r]
    [lll]!{\xcapv@(0)} [ur]!{\vcross}[r][ur]!{\xcapv@(0)}
    [lll]!{\xcapv@(0)} [ur]!{\xcapv@(0)}[ur]!{\vcap-}[dr]
}\]\caption*{$8_{15}$}\end{subfigure}
\begin{subfigure}{0.13\textwidth}\[\xygraph{ !{/r0.75pc/:}
    !{\xcapv@(0)}[ur]!{\xcapv@(0)}[ur][d]
    [ll]!{\vcross}[rr]!{\vcap}[r]
    [lll]!{\xcapv@(0)} [ur]!{\vtwist}[r][ur]!{\xcapv@(0)}
    [lll]!{\xcapv@(0)} [ur]!{\vtwist}[r][ur]!{\xcapv@(0)}
    [lll]!{\xcapv@(0)} [ur]!{\vtwist}[r][ur]!{\xcapv@(0)}
    [lll]!{\vcross}[r][ur]!{\xcapv@(0)}[ur]!{\xcapv@(0)}
    [lll]!{\xcapv@(0)} [ur]!{\vtwist}[r][ur]!{\xcapv@(0)}
    [lll]!{\xcapv@(0)} [ur]!{\xcapv@(0)}[ur]!{\vcross}[r]
    [lll]!{\xcapv@(0)} [ur]!{\vtwist}[r][ur]!{\xcapv@(0)}
    [lll]!{\xcapv@(0)} [ur]!{\vtwist}[r][ur]!{\xcapv@(0)}
    [lll]!{\xcapv@(0)} [ur]!{\xcapv@(0)} [ur]!{\vcap-}[dr]
}\]\caption*{$9_{22}$}\end{subfigure}
\begin{subfigure}{0.13\textwidth}\[\xygraph{ !{/r0.75pc/:}
    !{\xcapv@(0)} [ur]!{\xcapv@(0)}
    [l]!{\vtwist}
    !{\vtwist}[rr]!{\vcap}[r]
    [lll]!{\xcapv@(0)} [ur]!{\vcross}[r][ur]!{\xcapv@(0)}
    [lll]!{\xcapv@(0)} [ur]!{\vcross}[r][ur]!{\xcapv@(0)}
    [lll]!{\vtwist}[r][ur]!{\vtwist}[r]
    [lll]!{\xcapv@(0)} [ur]!{\vcross}[r][ur]!{\xcapv@(0)}
    [lll]!{\xcapv@(0)} [ur]!{\vcross}[r][ur]!{\xcapv@(0)}
    [lll]!{\xcapv@(0)} [ur]!{\vcross}[r][ur]!{\xcapv@(0)}
    [lll]!{\xcapv@(0)} [ur]!{\xcapv@(0)} [ur]!{\vcap-}[dr]
}\]\caption*{$9_{24}$}\end{subfigure}
\begin{subfigure}{0.13\textwidth}\[\xygraph{ !{/r0.75pc/:}
    !{\xcapv@(0)}[ur][d]!{\vcap}[r][ur]!{\xcapv@(0)}
    [lll]!{\vtwist}[r][ur]!{\vtwist}[r]
    [lll]!{\vtwist}[r][ur]!{\vtwist}[r]
    [lll]!{\xcapv@(0)} [ur]!{\vcross}[r][ur]!{\xcapv@(0)}
    [lll]!{\xcapv@(0)} [ur]!{\vcross}[r][ur]!{\xcapv@(0)}
    [lll]!{\xcapv@(0)} [ur]!{\xcapv@(0)}[ur]!{\vtwist}[r]
    [lll]!{\xcapv@(0)} [ur]!{\vcross}[r][ur]!{\xcapv@(0)}
    [lll]!{\vtwist}[r][ur]!{\xcapv@(0)}[ur]!{\xcapv@(0)}
    [lll]!{\xcapv@(0)} [ur]!{\vcap-}[dr][ur]!{\xcapv@(0)}
}\]\caption*{$9_{25}$}\end{subfigure}
\begin{subfigure}{0.13\textwidth}\[\xygraph{ !{/r0.75pc/:}
    !{\xcapv@(0)}[ur]!{\xcapv@(0)}
    [l]!{\vtwist}
    !{\vtwist}[rr]!{\vcap}[r]
    [lll]!{\xcapv@(0)} [ur]!{\vcross}[r][ur]!{\xcapv@(0)}
    [lll]!{\xcapv@(0)} [ur]!{\xcapv@(0)}[ur]!{\vtwist}[r]
    [lll]!{\xcapv@(0)} [ur]!{\vcross}[r][ur]!{\xcapv@(0)}
    [lll]!{\vtwist}[r][ur]!{\xcapv@(0)}[ur]!{\xcapv@(0)}
    [lll]!{\xcapv@(0)} [ur]!{\xcapv@(0)}[ur]!{\vtwist}[r]
    [lll]!{\xcapv@(0)} [ur]!{\vcross}[r][ur]!{\xcapv@(0)}
    [lll]!{\xcapv@(0)} [ur]!{\vcross}[r][ur]!{\xcapv@(0)}
    [lll]!{\xcapv@(0)}[ur]!{\xcapv@(0)}[ur]!{\vcap-}[dr]
}\]\caption*{$9_{28}$}\end{subfigure}
\begin{subfigure}{0.13\textwidth}\[\xygraph{ !{/r0.75pc/:}
    !{\xcapv@(0)}[ur]!{\xcapv@(0)}
    [l]!{\vcross}[r][ur][d]
    [ll]!{\vcross}[r][r]!{\vcap}[r]
    [lll]!{\xcapv@(0)} [ur]!{\vtwist}[r][ur]!{\xcapv@(0)}
    [lll]!{\vcross}[r][ur]!{\vcross}[r]
    [lll]!{\xcapv@(0)} [ur]!{\vtwist}[r][ur]!{\xcapv@(0)}
    [lll]!{\vcross}[r][ur]!{\xcapv@(0)}[ur]!{\xcapv@(0)}
    [lll]!{\vcross}[r][ur]!{\xcapv@(0)}[ur]!{\xcapv@(0)}
    [lll]!{\xcapv@(0)} [ur]!{\vtwist}[r][ur]!{\xcapv@(0)}
    [lll]!{\xcapv@(0)} [ur]!{\xcapv@(0)}[ur]!{\vcap-}[dr]
}\]\caption*{$9_{29}$}\end{subfigure}
\begin{subfigure}{0.13\textwidth}\[\xygraph{ !{/r0.75pc/:}
    !{\xcapv@(0)}[ur][d]!{\vcap}[r][ur]!{\xcapv@(0)}
    [lll]!{\xcapv@(0)} [ur]!{\xcapv@(0)}[ur]!{\vcross}[r]
    [lll]!{\xcapv@(0)} [ur]!{\vtwist}[r][ur]!{\xcapv@(0)}
    [lll]!{\vtwist}[r][ur]!{\vcross}[r]
    [lll]!{\xcapv@(0)} [ur]!{\vtwist}[r][ur]!{\xcapv@(0)}
    [lll]!{\vtwist}[r][ur]!{\vcross}[r]
    [lll]!{\xcapv@(0)} [ur]!{\vtwist}[r][ur]!{\xcapv@(0)}
    [lll]!{\vcross}[r][ur]!{\xcapv@(0)}[ur]!{\xcapv@(0)}
    [lll]!{\xcapv@(0)} [ur]!{\vtwist}[r][ur]!{\xcapv@(0)}
    [lll]!{\xcapv@(0)} [ur]!{\xcapv@(0)}[ur]!{\vcross}[r]
    [lll]!{\xcapv@(0)} [ur]!{\vcap-}[dr][ur]!{\xcapv@(0)}
}\]\caption*{$9_{30}$}\end{subfigure}
\begin{subfigure}{0.13\textwidth}\[\xygraph{ !{/r0.75pc/:}
    !{\xcapv@(0)}[ur]!{\xcapv@(0)}
    [l]!{\vcross}[r]
    [l]!{\vcross}[r][r]!{\vcap}[r]
    [lll]!{\xcapv@(0)} [ur]!{\vtwist}[r][ur]!{\xcapv@(0)}
    [lll]!{\xcapv@(0)} [ur]!{\vtwist}[r][ur]!{\xcapv@(0)}
    [lll]!{\vcross}[r][ur]!{\vcross}[r]
    [lll]!{\xcapv@(0)} [ur]!{\vtwist}[r][ur]!{\xcapv@(0)}
    [lll]!{\vcross}[r][ur]!{\xcapv@(0)}[ur]!{\xcapv@(0)}
    [lll]!{\xcapv@(0)} [ur]!{\vtwist}[r][ur]!{\xcapv@(0)}
    [lll]!{\xcapv@(0)} [ur]!{\xcapv@(0)}[ur]!{\vcap-}[dr]
}\]\caption*{$9_{32}$}\end{subfigure}
\begin{subfigure}{0.13\textwidth}\[\xygraph{ !{/r0.75pc/:}
    !{\xcapv@(0)}[ur]!{\xcapv@(0)}
    [l]!{\vtwist}[r][r]!{\vcap}[r]
    [lll]!{\xcapv@(0)} [ur]!{\vcross}[r][ur]!{\xcapv@(0)}
    [lll]!{\vtwist}[r] [ur]!{\vtwist}[r]
    [lll]!{\xcapv@(0)} [ur]!{\vcross}[r][ur]!{\xcapv@(0)}
    [lll]!{\vtwist}[r][ur]!{\xcapv@(0)}[ur]!{\xcapv@(0)}
    [lll]!{\xcapv@(0)} [ur]!{\vcross}[r][ur]!{\xcapv@(0)}
    [lll]!{\xcapv@(0)} [ur]!{\xcapv@(0)}[ur]!{\vtwist}[r]
    [lll]!{\xcapv@(0)} [ur]!{\vcross}[r][ur]!{\xcapv@(0)}
    [lll]!{\xcapv@(0)} [ur]!{\xcapv@(0)}[ur]!{\vcap-}[dr]
}\]\caption*{$9_{33}$}\end{subfigure}
\begin{subfigure}{0.13\textwidth}\[\xygraph{ !{/r0.75pc/:}
    !{\xcapv@(0)}[ur]!{\xcapv@(0)}
    [l]!{\vcross}[rr]!{\vcap}[r]
    [lll]!{\xcapv@(0)} [ur]!{\vtwist}[r][ur]!{\xcapv@(0)}
    [lll]!{\vcross}[r][ur]!{\xcapv@(0)}[ur]!{\xcapv@(0)}
    [lll]!{\xcapv@(0)} [ur]!{\vtwist}[r][ur]!{\xcapv@(0)}
    [lll]!{\xcapv@(0)} [ur]!{\xcapv@(0)}[ur]!{\vcross}[r]
    [lll]!{\xcapv@(0)} [ur]!{\vtwist}[r][ur]!{\xcapv@(0)}
    [lll]!{\vcross}[r][ur]!{\xcapv@(0)}[ur]!{\xcapv@(0)}
    [lll]!{\xcapv@(0)} [ur]!{\vtwist}[r][ur]!{\xcapv@(0)}
    [lll]!{\xcapv@(0)} [ur]!{\vtwist}[r][ur]!{\xcapv@(0)}
    [lll]!{\xcapv@(0)} [ur]!{\xcapv@(0)}[ur]!{\vcap-}[dr]
}\]\caption*{$9_{34}$}\end{subfigure}
\begin{subfigure}{0.13\textwidth}\[\xygraph{ !{/r0.75pc/:}
    !{\xcapv@(0)}[ur][d]!{\vcap}[r][ur]!{\xcapv@(0)}
    [lll]!{\vtwist}[r][ur]!{\xcapv@(0)}[ur]!{\xcapv@(0)}
    [lll]!{\xcapv@(0)} [ur]!{\vcross}[r][ur]!{\xcapv@(0)}
    [lll]!{\xcapv@(0)} [ur]!{\xcapv@(0)}[ur]!{\vtwist}[r]
    [lll]!{\xcapv@(0)} [ur]!{\vcross}[r][ur]!{\xcapv@(0)}
    [lll]!{\xcapv@(0)} [ur]!{\vcross}[r][ur]!{\xcapv@(0)}
    [lll]!{\xcapv@(0)} [ur]!{\vcross}[r][ur]!{\xcapv@(0)}
    [lll]!{\vtwist}[r][ur]!{\xcapv@(0)}[ur]!{\xcapv@(0)}
    [lll]!{\xcapv@(0)} [ur]!{\vcross}[r][ur]!{\xcapv@(0)}
    [lll]!{\xcapv@(0)} [ur]!{\xcapv@(0)}[ur]!{\vtwist}[r]
    [lll]!{\xcapv@(0)} [ur]!{\vcap-}[dr][ur]!{\xcapv@(0)}
}\]\caption*{$9_{35}$}\end{subfigure}
\begin{subfigure}{0.13\textwidth}\[\xygraph{ !{/r0.75pc/:}
    !{\xcapv@(0)}[ur][d]!{\vcap}[r][ur]!{\xcapv@(0)}
    [lll]!{\xcapv@(0)} [ur]!{\xcapv@(0)}[ur]!{\vtwist}[r]
    [lll]!{\xcapv@(0)} [ur]!{\xcapv@(0)}[ur]!{\vtwist}[r]
    [lll]!{\xcapv@(0)} [ur]!{\vcross}[r][ur]!{\xcapv@(0)}
    [lll]!{\vtwist}[r][ur]!{\xcapv@(0)}[ur]!{\xcapv@(0)}
    [lll]!{\xcapv@(0)} [ur]!{\vcross}[r][ur]!{\xcapv@(0)}
    [lll]!{\xcapv@(0)} [ur]!{\vcross}[r][ur]!{\xcapv@(0)}
    [lll]!{\xcapv@(0)} [ur]!{\xcapv@(0)}[ur]!{\vtwist}[r]
    [lll]!{\xcapv@(0)} [ur]!{\vcross}[r][ur]!{\xcapv@(0)}
    [lll]!{\vtwist}[r][ur]!{\xcapv@(0)}[ur]!{\xcapv@(0)}
    [lll]!{\xcapv@(0)} [ur]!{\vcap-}[dr][ur]!{\xcapv@(0)}
}\]\caption*{$9_{36}$}\end{subfigure}
\begin{subfigure}{0.13\textwidth}\[\xygraph{ !{/r0.75pc/:}
    !{\xcapv@(0)}[ur]!{\xcapv@(0)}[ur][d]!{\vcap}[r]
    [lll]!{\xcapv@(0)} [ur]!{\vtwist}[r][ur]!{\xcapv@(0)}
    [lll]!{\xcapv@(0)} [ur]!{\vtwist}[r][ur]!{\xcapv@(0)}
    [lll]!{\vcross}[r][ur]!{\vcross}[r]
    [lll]!{\xcapv@(0)} [ur]!{\vtwist}[r][ur]!{\xcapv@(0)}
    [lll]!{\vcross}[r][ur]!{\vcross}[r]
    [lll]!{\xcapv@(0)} [ur]!{\xcapv@(0)}[ur]!{\vcross}[r]
    [lll]!{\xcapv@(0)} [ur]!{\vtwist}[r][ur]!{\xcapv@(0)}
    [lll]!{\xcapv@(0)} [ur]!{\xcapv@(0)}[ur]!{\vcap-}[dr]
}\]\caption*{$9_{37}$}\end{subfigure}
\begin{subfigure}{0.13\textwidth}\[\xygraph{ !{/r0.75pc/:}
    !{\xcapv@(0)}[ur]!{\xcapv@(0)}
    [l]!{\vtwist}[r][d]
    [lll]!{\vcap}[urr]!{\vtwist}[r]
    [lll]!{\xcapv@(0)} [ur]!{\vcross}[r][ur]!{\xcapv@(0)}
    [lll]!{\xcapv@(0)} [ur]!{\vcross}[r][ur]!{\xcapv@(0)}
    [lll]!{\xcapv@(0)} [ur]!{\xcapv@(0)}[ur]!{\vtwist}[r]
    [lll]!{\xcapv@(0)} [ur]!{\vcross}[r][ur]!{\xcapv@(0)}
    [lll]!{\vtwist}[r][ur]!{\xcapv@(0)}[ur]!{\xcapv@(0)}
    [lll]!{\xcapv@(0)} [ur]!{\vcross}[r][ur]!{\xcapv@(0)}
    [lll]!{\xcapv@(0)} [ur]!{\xcapv@(0)}[ur]!{\vtwist}[r]
    [lll]!{\xcapv@(0)} [ur]!{\vcap-}[dr][ur]!{\xcapv@(0)}
}\]\caption*{$9_{38}$}\end{subfigure}
\begin{subfigure}{0.13\textwidth}\[\xygraph{ !{/r0.75pc/:}
    !{\xcapv@(0)}[r]!{\vcap}[ur][r]!{\xcapv@(0)}
    [lll]!{\vcross}[r][ur]!{\xcapv@(0)}[ur]!{\xcapv@(0)}
    [lll]!{\xcapv@(0)}[ur]!{\vtwist}[r][ur]!{\xcapv@(0)}
    [lll]!{\xcapv@(0)}[ur]!{\xcapv@(0)}[ur]!{\vcross}[r]
    [lll]!{\xcapv@(0)}[ur]!{\vtwist}[r][ur]!{\xcapv@(0)}
    [lll]!{\vcross}[r][ur]!{\xcapv@(0)}[ur]!{\xcapv@(0)}
    [lll]!{\xcapv@(0)}[ur]!{\vtwist}[r][ur]!{\xcapv@(0)}
    [lll]!{\vcross}[r][ur]!{\vcross}[r]
    [lll]!{\vcross}[r][ur]!{\xcapv@(0)}[ur]!{\xcapv@(0)}
    [lll]!{\xcapv@(0)}[ur]!{\vcap-}[r][r]!{\xcapv@(0)}
}\]\caption*{$9_{39}$}\end{subfigure}
\begin{subfigure}{0.13\textwidth}\[\xygraph{ !{/r0.75pc/:}
    !{\xcapv@(0)}[ur][d]!{\vcap}[ur][r]!{\xcapv@(0)}
    [lll]!{\vcross}[r][ur]!{\vcross}[r]
    [lll]!{\xcapv@(0)} [ur]!{\vtwist}[r][ur]!{\xcapv@(0)}
    [lll]!{\vcross}[r][ur]!{\xcapv@(0)}[ur]!{\xcapv@(0)}
    [lll]!{\xcapv@(0)} [ur]!{\vtwist}[r][ur]!{\xcapv@(0)}
    [lll]!{\xcapv@(0)} [ur]!{\xcapv@(0)}[ur]!{\vcross}[r]
    [lll]!{\xcapv@(0)} [ur]!{\vtwist}[r][ur]!{\xcapv@(0)}
    [lll]!{\vcross}[r][ur]!{\vcross}[r]
    [lll]!{\xcapv@(0)} [ur]!{\vcap-}[dr][ur]!{\xcapv@(0)}
}\]\caption*{$9_{40}$}\end{subfigure}
\begin{subfigure}{0.13\textwidth}\[\xygraph{ !{/r0.75pc/:}    
    !{\xcapv@(0)}[ur]!{\xcapv@(0)}[ur][d]!{\vcap}[r]
     [lll]!{\xcapv@(0)} [ur]!{\vcross}[r][ur]!{\xcapv@(0)}
     [lll]!{\vcross}[r][ur]!{\vtwist}[r]
     [lll]!{\vcross}[r][ur]!{\xcapv@(0)}[ur]!{\xcapv@(0)}
     [lll]!{\xcapv@(0)} [ur]!{\vcross}[r][ur]!{\xcapv@(0)}
     [lll]!{\xcapv@(0)} [ur]!{\xcapv@(0)}[ur]!{\vtwist}[r]
     [lll]!{\xcapv@(0)} [ur]!{\vcross}[r][ur]!{\xcapv@(0)}
     [lll]!{\vtwist}[r][ur]!{\vtwist}[r]
     [lll]!{\xcapv@(0)} [ur]!{\vcross}[r][ur]!{\xcapv@(0)}
     [lll]!{\xcapv@(0)} [ur]!{\xcapv@(0)}[ur]!{\vcap-}[dr]
}\]\caption*{$9_{41}$}\end{subfigure}
\begin{subfigure}{0.13\textwidth}\[\xygraph{ !{/r0.75pc/:}    
    !{\xcapv@(0)}[ur][d]!{\vcap}[r][ur]!{\xcapv@(0)}
     [lll]!{\vcross}[r][ur]!{\vtwist}[r]
     [lll]!{\vcross}[r][ur]!{\xcapv@(0)}[ur]!{\xcapv@(0)}
     [lll]!{\xcapv@(0)} [ur]!{\vtwist}[r][ur]!{\xcapv@(0)}
     [lll]!{\xcapv@(0)} [ur]!{\vtwist}[r][ur]!{\xcapv@(0)}
     [lll]!{\xcapv@(0)} [ur]!{\vtwist}[r][ur]!{\xcapv@(0)}
     [lll]!{\vcross}[r][ur]!{\vcross}[r]
     [lll]!{\vcross}[r][ur]!{\vcross}[r]
     [lll]!{\vcross}[r][ur]!{\xcapv@(0)}[ur]!{\xcapv@(0)}
     [lll]!{\xcapv@(0)} [ur]!{\vcap-}[r][r]!{\xcapv@(0)}
}\]\caption*{$9_{42}$}\end{subfigure}
\begin{subfigure}{0.13\textwidth}\[\xygraph{ !{/r0.75pc/:}
     !{\xcapv@(0)}[ur]!{\xcapv@(0)}
     [l]!{\vtwist}[r][d]
     [lll]!{\vcap}[r][ur]!{\vtwist}[r]
     [lll]!{\xcapv@(0)}[ur]!{\vcross}[r][ur]!{\xcapv@(0)}
     [lll]!{\vtwist}[r][ur]!{\xcapv@(0)}[ur]!{\xcapv@(0)}
     [lll]!{\xcapv@(0)} [ur]!{\vcross}[r][ur]!{\xcapv@(0)}
     [lll]!{\vtwist}[r][ur]!{\vcross}[r]
     [lll]!{\xcapv@(0)} [ur]!{\vtwist}[r][ur]!{\xcapv@(0)}
     [lll]!{\xcapv@(0)} [ur]!{\vtwist}[r][ur]!{\xcapv@(0)}
     [lll]!{\xcapv@(0)} [ur]!{\xcapv@(0)}[ur]!{\vcap-}
}\]\caption*{$9_{43}$}\end{subfigure}
\begin{subfigure}{0.13\textwidth}\[\xygraph{ !{/r0.75pc/:}
     !{\xcapv@(0)}[ur]!{\xcapv@(0)}
     [l]!{\vcross}[r][d]
     [lll]!{\vcap}[r][ur]!{\vcross}[r]
     [lll]!{\xcapv@(0)} [ur]!{\vcross}[r][ur]!{\xcapv@(0)}
     [lll]!{\vtwist}[r][ur]!{\xcapv@(0)}[ur]!{\xcapv@(0)}
     [lll]!{\xcapv@(0)} [ur]!{\vcross}[r][ur]!{\xcapv@(0)}
     [lll]!{\vtwist}[r][ur]!{\vcross}[r]
     [lll]!{\xcapv@(0)} [ur]!{\vtwist}[r][ur]!{\xcapv@(0)}
     [lll]!{\xcapv@(0)} [ur]!{\vtwist}[r][ur]!{\xcapv@(0)}
     [lll]!{\xcapv@(0)} [ur]!{\xcapv@(0)}[ur]!{\vcap-}
}\]\caption*{$9_{44}$}\end{subfigure}
\begin{subfigure}{0.13\textwidth}\[\xygraph{ !{/r0.75pc/:}
     !{\xcapv@(0)}[ur]!{\xcapv@(0)}
     [l]!{\vcross}[d]
     [ll]!{\vcap}[r][ur]!{\vcross}
     [ll]!{\xcapv@(0)} [ur]!{\vcross}[r][ur]!{\xcapv@(0)}
     [lll]!{\vtwist}[r][ur]!{\xcapv@(0)}[ur]!{\xcapv@(0)}
     [lll]!{\xcapv@(0)} [ur]!{\vcross}[r][ur]!{\xcapv@(0)}
     [lll]!{\vtwist}[r][ur]!{\vtwist}[r]
     [lll]!{\xcapv@(0)} [ur]!{\vcross}[r][ur]!{\xcapv@(0)}
     [lll]!{\xcapv@(0)} [ur]!{\vcross}[r][ur]!{\xcapv@(0)}
     [lll]!{\xcapv@(0)} [ur]!{\xcapv@(0)}[ur]!{\vcap-}
}\]\caption*{$9_{45}$}\end{subfigure}
\begin{subfigure}{0.13\textwidth}\[\xygraph{ !{/r0.75pc/:}
    !{\xcapv@(0)}[ur][d]!{\vcap}[r][ur]!{\xcapv@(0)}
     [lll]!{\xcapv@(0)} [ur]!{\xcapv@(0)}[ur]!{\vtwist}[r]
     [lll]!{\xcapv@(0)} [ur]!{\vcross}[r][ur]!{\xcapv@(0)}
     [lll]!{\vcross}[r][ur]!{\vtwist}[r]
     [lll]!{\xcapv@(0)} [ur]!{\vtwist}[r][ur]!{\xcapv@(0)}
     [lll]!{\vcross}[r][ur]!{\vtwist}[r]
     [lll]!{\xcapv@(0)} [ur]!{\xcapv@(0)}[ur]!{\vtwist}[r]
     [lll]!{\xcapv@(0)} [ur]!{\xcapv@(0)}[ur]!{\vtwist}[r]
     [lll]!{\xcapv@(0)} [ur]!{\vcap-}[dr][ur]!{\xcapv@(0)}
}\]\caption*{$9_{46}$}\end{subfigure}
\begin{subfigure}{0.13\textwidth}\[\xygraph{ !{/r0.75pc/:}
    !{\xcapv@(0)}[ur][d]!{\vcap}[r][ur]!{\xcapv@(0)}
     [lll]!{\vtwist}[r][ur]!{\vtwist}[r]
     [lll]!{\xcapv@(0)} [ur]!{\vcross}[r][ur]!{\xcapv@(0)}
     [lll]!{\vtwist}[r][ur]!{\vtwist}[r]
     [lll]!{\xcapv@(0)} [ur]!{\vcross}[r][ur]!{\xcapv@(0)}
     [lll]!{\vcross}[r][ur]!{\vtwist}[r]
     [lll]!{\vcross}[r][ur]!{\xcapv@(0)}[ur]!{\xcapv@(0)}
     [lll]!{\xcapv@(0)} [ur]!{\vcap-}[dr][ur]!{\xcapv@(0)}
}\]\caption*{$9_{47}$}\end{subfigure}
\begin{subfigure}{0.13\textwidth}\[\xygraph{ !{/r0.75pc/:}
    !{\xcapv@(0)}[ur][d]!{\vcap}[r][ur]!{\xcapv@(0)}
     [lll]!{\xcapv@(0)} [ur]!{\xcapv@(0)}[ur]!{\vcross}[r]
     [lll]!{\xcapv@(0)} [ur]!{\vtwist}[r][ur]!{\xcapv@(0)}
     [lll]!{\vcross}[r][ur]!{\vcross}[r]
     [lll]!{\xcapv@(0)} [ur]!{\vtwist}[r][ur]!{\xcapv@(0)}
     [lll]!{\vcross}[r][ur]!{\vtwist}[r]
     [lll]!{\xcapv@(0)} [ur]!{\xcapv@(0)}[ur]!{\vtwist}[r]
     [lll]!{\xcapv@(0)} [ur]!{\xcapv@(0)}[ur]!{\vtwist}[r]
     [lll]!{\xcapv@(0)} [ur]!{\vcap-}[dr][ur]!{\xcapv@(0)}
}\]\caption*{$9_{48}$}\end{subfigure}
\begin{subfigure}{0.13\textwidth}\[\xygraph{ !{/r0.75pc/:}
    !{\xcapv@(0)}[r]!{\vcap}[urr]!{\xcapv@(0)}
     [lll]!{\xcapv@(0)} [ur]!{\xcapv@(0)}[ur]!{\vtwist}[r]
     [lll]!{\xcapv@(0)} [ur]!{\vcross}[r][ur]!{\xcapv@(0)}
     [lll]!{\vtwist}[r][ur]!{\xcapv@(0)}[ur]!{\xcapv@(0)}
     [lll]!{\xcapv@(0)} [ur]!{\vcross}[r][ur]!{\xcapv@(0)}
     [lll]!{\xcapv@(0)} [ur]!{\vcross}[r][ur]!{\xcapv@(0)}
     [lll]!{\xcapv@(0)} [ur]!{\xcapv@(0)}[ur]!{\vcross}[r]
     [lll]!{\xcapv@(0)} [ur]!{\vtwist}[r][ur]!{\xcapv@(0)}
     [lll]!{\vtwist}[r][ur]!{\vcross}[r]
     [lll]!{\xcapv@(0)} [ur]!{\vcap-}[rr]!{\xcapv@(0)}
}\]\caption*{$9_{49}$}\end{subfigure}
    \caption{The $(1,2)$-representatives of the booklink types required to complete Table~\ref{table:knot_spectra}.}\label{fig:missing}
\end{figure}

\bibliographystyle{alpha}
\bibliography{./biblio}

\end{document}